\documentclass[paper=a4,twoside=false,fontsize=11pt, abstract=true]{scrartcl}
\usepackage{amsmath,amsthm,amsfonts,nicefrac,mathtools}
\usepackage[utf8]{inputenc}
\usepackage[T1]{fontenc}
\usepackage{lmodern}
\usepackage{microtype}
\usepackage[final,colorlinks]{hyperref}
\usepackage{fixme}
\fxusetheme{color}
\usepackage{enumitem}
\usepackage[text={15cm,23cm},centering]{geometry}
\newcommand{\abs}[1]{\left|#1\right|}

\usepackage{gitinfo}
\usepackage{lastpage}
\usepackage[automark]{scrpage2}

\clearscrheadings
\newcommand{\xR}{\mathbb{R}}
\newcommand{\xN}{\mathbb{N}}
\newcommand{\xZ}{\mathbb{Z}}

\newcommand{\tX}{\tilde{X}}
\newcommand{\tY}{\tilde{Y}}
\newcommand{\tZ}{\tilde{Z}}

\newcommand{\tx}{\tilde{x}}
\newcommand{\ty}{\tilde{y}}
\let\tXN\tXn
\let\tYN\tYn
\newcommand{\tu}{\tilde{u}}
\newcommand{\tq}{\tilde{q}}
\newcommand{\tT}{\tilde{T}}

\newcommand{\bigO}{\mathcal{O}}
\newcommand{\ind}[1]{\mathbf{1}_{#1}}

\newcommand{\prb}[1]{\mathbb{P}\left[#1\right]}
\newcommand{\esp}[1]{\mathbb{E}\left[#1\right]}
\newcommand{\espX}[1]{\mathbb{E}_x\left[#1\right]}
\newcommand{\floor}[1]{\lfloor#1\rfloor}
\newcommand{\SET}[1]{\left\{#1\right\}}
\newcommand{\PAR}[1]{\left(#1\right)}
\newcommand{\cov}[1]{\mathbf{Cov}\left(#1\right)}
\newcommand{\var}[1]{\mathbf{Var}\left(#1\right)}
\newcommand{\varX}[1]{\mathbf{Var}_x\left(#1\right)}
\newcommand{\covX}[1]{\mathbf{Cov}_x\left(#1\right)}

\newcommand{\propNorm}{\mathbf{x}_N}
\newcommand{\stateSpace}{\mathcal{S}_n}
\newtheorem{thrm}{Theorem}
\newtheorem{proposition}[thrm]{Proposition}
\newtheorem{lmm}[thrm]{Lemma}
\newtheorem{crllr}[thrm]{Corollary}
\newtheorem{dfn}[thrm]{Definition}
\newtheorem{remark}[thrm]{Remark}

\title{A Wright-Fisher model with indirect selection}
\author{Ludovic Goudenège and Pierre-André Zitt}
\begin{document}

\maketitle
\begin{abstract}%
We study a generalization of the Wright--Fisher model in which some individuals
adopt a behavior that is harmful to others without any direct advantage for themselves.
This model is motivated by studies of spiteful behavior in nature, including
several species of parasitoid hymenoptera in which sperm-depleted males continue to mate de-
spite not being fertile.  

We first study a single reproductive season, then use it as a building block for a generalized
Wright--Fisher model. In the large population limit, for male-skewed sex ratios,
we rigorously derive the convergence of the renormalized process to a diffusion with a frequency-dependent selection and genetic drift.
This allows a quantitative comparison of the indirect
selective advantage with the direct one classically considered in the Wright--Fisher model.   
  
From the mathematical point of view, each season is modeled by a mix
between samplings with and
without replacement, and analyzed by a sort of ``reverse numerical
analysis'', viewing a key recurrence relation as a discretization scheme for
a PDE. The diffusion approximation is then obtained by classical methods.

  \bigskip


\medskip

\end{abstract}

\section{Introduction: models and main results}

\subsection{Harmful behaviours and population genetics}
The object of population genetics is to understand how the genetic composition
of a population changes through time in response to mutation, natural selection and
demographic stochasticity (``genetic drift'') and mutations. In the simplest case,
consider a gene with a haploid locus segregating two alleles (say ``white'' and ``black''),
which affect an individual's phenotype. Here we are
interested in the changes of the
proportion of individuals carrying the white allele over generations. 
One of the simplest stochastic models for this evolution is 
the classical Wright--Fisher model for genetic drift%
\footnote{%
  Let us recall the unfortunate polysemy of the word ``drift''. 
  In the biological literature ``genetic drift'' corresponds to the noise-induced variations.
  When using a stochastic model, this is at odds with the 
  ``drift'' of a diffusion, i.e.~the first order term that 
models a deterministic force.} %
(its precise definition is recalled below in Section~\ref{sec:recallWF}). 
It is simple enough that a very detailed mathematical analysis can be performed
(see for example the monographs \cite{Dur08} or \cite{Eth11}, where many other questions 
and models are studied from a mathematical point of view). 
Many variations of this model have been studied, adding selection 
and mutation to the picture. Classically, selection has been added to this model by 
stipulating that one of the alleles is $(1+\beta)$ 
more likely to be chosen for the next generation
than the other allele. 
This models a direct advantage: for example, the eggs carrying the white
allele may have more chance to mature. 

In many biological settings, individuals perform actions that may harm others
without giving the perpetrator any direct advantage. For example, males of
several invertebrate species  have a limited sperm stock. Surprisingly, they
have been reported to continue to attempt mating with virgin females while
being completely sperm depleted \cite{Damiens2006Why, Steiner2008Mating}.
Obviously, this behaviour does not aim to fertilize the eggs of these virgin
females. However, in these species, copulation (with or without sperm release)
has the property of stopping female sexual receptivity. This, for instance can
occur as a behavioural response of the female or as a consequence of toxic
seminal fluids or plugs inserted in female genitalia by males
\cite{Rice1996Sexually, Radhakrishnan2009Multiple}. Males can also guard the
female during her receptivity period without copulating with her. These male
behaviours do not increase the absolute number of eggs they fertilize.
However, because these actions limit the ability of other males to
fertilize eggs, it has been suggested that they may have evolved as a male
mating strategy to increase the relative number of offspring sired by individuals
that use this strategy. This model is inspired by a model
for the evolution of spiteful behaviour (see, e.g., \cite{Ham70}, \cite{Dio07}
and \cite{FWR01} for discussions of Hamiltonian spite).

Our aim is to analyze a variation of the Wright--Fisher model
where such an effect appears. Quite interestingly, this model will prove to 
be equivalent (in the large population limit) to a model with frequency dependent selection. 

In the remainder of this introduction we first define a model for one generation, 
where a certain number of females visit a pool of males, 
some of which carry the black allele that codes for the ``harmful'' behaviour. 
When the number of individuals is large we can analyze precisely the
reproduction probabilities for each type of individual. Finally
we show how to adapt the Wright--Fisher model to our case, and state
our main result, namely a diffusion limit for the renormalized multi-generation model. 

\subsection{Basic model}
\label{sec:basic_model}
In the basic model, suggested by F.-X.~Dechaume-Moncharmont and M.~Galipaud
\footnote{Personal communication.},  consider an urn with~$w$ white balls
and~$b$ black balls.  All balls begin as ``unmarked''. Draw~$f$ times from this
urn, with the following rule:
\begin{itemize}
	\item if the ball drawn is white, mark it and remove it from the urn;
	\item if it is black and unmarked, mark it and put it back in the urn;
	\item if it is black and already marked, put it back in the urn. 
\end{itemize}
After the $f$ draws, call~$X$ the number of
marked white balls and~$Y$ the number of marked black balls. 

This models a reproductive season. The balls represent males, and each draw
corresponds to a reproduction attempt by a different female. The marks
represent a successful reproduction.  The white balls ``play fair'': it they
are chosen by a female, they reproduce and retire from the game.  The black
balls, even after reproduction, ``stay in the game'': they may be chosen again
in subsequent draws. Even if it is chosen multiple times, a black ball only
reproduces once, so that black balls do not get a direct reproductive advantage
from their behaviour.  In particular, if the colors of all the other balls are
fixed, the probability of reproduction does not depend on the ball's color.
However, the black balls ``harm'' all the other balls, possibly depriving them
of reproduction attempts. 
The variables~$X$ and~$Y$ count the number of white/black males that have reproduced. 

\begin{remark}
  [Simplification] This model is of  course very simplified. In particular
  the mating phenotype of the males only depends on a haploid locus which is
  paternally inherited; this assumption does not hold for example for male
  hymenopteran parasitoids, who inherit their genomes from their mothers. 
  For simplicity we restrict ourselves to one model, keeping in mind that
  other models may lead to different expressions of the drift and 
  variance for the diffusion limit. 
\end{remark}

To compare the two strategies, we begin by comparing two individuals. In an urn 
with $w$ white balls and $b$ black balls, we look at one particular white ball
(Walt) and one particular black ball (Bob). Define the probabilities of 
successful reproduction by:
\begin{align*}
  p_w(w,b,f) &= \prb{\text{Walt is chosen at least once in the $f$ draws}} \\
  p_b(w,b,f) &= \prb{\text{Bob is chosen at least once in the $f$ draws}} 
\end{align*}

\begin{thrm}
  \label{thm:avantage}
  The ``harmful'' males have a fitness advantage, in the sense that:
  \[ p_b(w,b,f) \geq p_w(w,b,f).\]
  The inequality is strict if $f\geq 2$ and $w,b\geq 1$. 
\end{thrm}

\subsection{Large population limit}
To quantify the advantage given by the ``harmful'' behaviour,
it is natural to look at a large population 
limit, when the number of black balls (``harmful'' males), white balls (regular males)
and the number of draws ($f$ i.e. females) go to infinity, 
while the respective proportions converge. We can describe the limiting behaviour 
of $p_b$ and $p_w$, and more importantly of the difference $p_b - p_w$, 
in terms of the solution $v$ of a specific PDE. 
To define this function $v$ and state the approximation result 
we need additional notation.  The numbers of individuals
$(w,b,f)$ will correspond in the continuous limit to proportions $(x,y,z)$ 
in  the set:
\[ \Omega = \left\{ (x,y,z)\in\xR_+ : x+y+z \leq 1\right\}.\]
For $(x,y,z) \in \Omega$, with $y>0$, we will prove
below (see Theorem~\ref{thm:propertiesOfU}) that
the equation
\begin{equation}
  x(1-e^{-t}) +  yt = z
  \label{eq=defT}
\end{equation}
has a unique solution $T(x,y,z)\in(0,\infty)$.
Define two functions $u$ and $v$ on $\Omega$ by:
\begin{align}
  u(x,y,z) &= \exp\left( - T(x,y,z) \right), &
  v(x,y,z) &= 1 - u(x,y,z).
  \label{eq=defUv}
\end{align}
A heuristic derivation of the expression of $u$, $v$ and $T$ will be
given below in Remark~\ref{rem:jay}. 

For any ``population size'' $N$ we will consider functions defined on the following 
discretization of $\Omega$~:
 \begin{align*}
   \Omega_N &= \left\{(w,b,f)\in\xZ_+^3 : w+b+f \leq N\right\}.
 \end{align*}
For any function $g:\Omega\to \xR$, we denote by $g^N$ the discretization
\begin{equation}
\begin{aligned}
  g^N : \Omega_N &\to \xR \\
        (w,b,f)  &\mapsto g\left( \frac{w}{N}, \frac{b}{N}, \frac{f}{N} \right).
\end{aligned}
  \label{eq:defDiscretization}
\end{equation}
If $p$ is a function on $\Omega_N$, 
we denote by $\delta_xp$, $\delta_yp$ the discrete differences:
\begin{align}
  \delta_x p(w,b,f) &= p(w+1,b,f) - p(w,b,f) &
  \delta_y p(w,b,f) &= p(w,b+1,f) - p(w,b,f). 
  \label{eq=defDeltaX}
\end{align}
Finally, most of the bounds we prove are uniform on specific subsets of $\Omega$ or $\Omega_N$. 
For any $y_0>0$, and any $s<1$, we define:
\begin{align*}
  \Omega (y_0)   &= \SET{ (x,y,z)\in \Omega:  y\geq y_0};\\
  \Omega_N (y_0) &= \SET{ (w,b,f)\in\xN^3 :  \PAR{\frac{w}{N}, \frac{b}{N}, \frac{f}{N}} \in \Omega(y_0)} ; \\
  \Omega(s)      &= \SET{ (x,y,z) \in \Omega:  z\leq s(x+y) \text{ and } x-z \geq (1-s)/(2+2s) }           ; \\
  \Omega_N(s)    &= \SET{ (w,b,f)\in \xN^3 : \PAR{\frac{w}{N}, \frac{b}{N}, \frac{f}{N}} \in \Omega(s)}.
\end{align*}
\begin{remark}[On the sets $\Omega(y_0)$ and $\Omega(s)$]
  Let us repeat that $x$, $y$ and $z$ are the continuous analogues of $w$,
  $b$ and $f$.  In this light, $y_0$ corresponds to a minimal proportion of
  ``harmful'' males, and $s$ to a maximal sex ratio. The second condition
  appearing in the definition of $\Omega(s)$ is less natural: it is a way of
  ruling out degenerate points where both $x$ and $y$ are small, which will be
  crucial for finding good bounds on $u$, $v$ and their derivatives (\emph{cf.}
  Theorem~\ref{thm:propertiesOfU}), while keeping an essential ``stability''
  property (\emph{cf.} the proof of the controls of errors at the end of
  Section~\ref{sec:convergence}). 
\end{remark}

Now we can state the first asymptotic result.

\begin{thrm}
  \label{thm=approximation}
  For any $y_0 >0$, there exists a constant $C(y_0)$ such that for all $N$,
  \[
    \forall (w,b,f)\in \Omega_N(y_0), 
    \quad
    \begin{aligned}
    \abs{ p_w(w,b,f) - v^N(w,b,f)} \leq \frac{C(y_0)}{N}, \\
    \abs{ p_b(w,b,f) - v^N(w,b,f)} \leq \frac{C(y_0)}{N},
  \end{aligned}
\]
where $v^N$ is the discretization of $v$ (see \eqref{eq=defUv} and \eqref{eq:defDiscretization}). 
  Moreover, the difference of fitness is of order $1/N$, and more precisely:
  \begin{align}
    \label{eq=cvg_differences}
    \forall (w,b,f) \in \Omega_N(y_0), \quad
    \abs{ p_b(w,b,f) - p_w(w,b,f) - \frac{1}{N}(\partial_xv - \partial_yv)^N(w,b,f)} \leq \frac{C(y_0)}{N^2}.
  \end{align}
  For any $s<1$, the same bounds hold uniformly on all $\Omega_N(s)$, 
  with $C(y_0)$ replaced by a constant $C(s)$ that only depends on $s$. 
\end{thrm}

\subsection{Multiple generations: the classical Wright--Fisher model with selection}
\label{sec:recallWF}

The Wright--Fisher model with selection is a Markov chain $(X^N_k)_{k\in\xN}$ on
$\{0, 1/N, 2/N, \dots 1\}$ that describes (a simplification of) the evolution of the
frequency of an allele in a population across generations. This is a very simplified model, 
where the size $N$ of the population is fixed. See for example the monographs
\cite{Dur08,Eth11} for a much more detailed exposition; we follow here
\cite{Eth11}, Section 5.2. To simplify the exposition 
suppose that the first allele is ``white'' and the second ``black''; at time $k$ 
a proportion $X^N_k$ of the population is ``white''. 
Given the state $x$ at time $k$, the next state
is chosen in the following way. 
\begin{description}
  \item[First step.] All individuals lay a very large number $M$ of eggs. A
    proportion $s_b(N)$ (resp.\ $s_w(N)$) of black
    (resp.\ white) eggs survive this first step,
    so there are $M\cdot N(1-x)\cdot s_b(N)$ black eggs and
    $M\cdot Nx \cdot s_w(N) $ white ones. 
  \item[Second step.] The population at time $k+1$, of size $N$, is chosen by picking
    randomly $N$ eggs among the surviving ones. Since 
    $M$ is very large, the number of white individuals at time $k+1$ is approximately
    binomial. If  the ratio of the 
    surviving probabilities is
    \[ 1 + \beta(N) = s_w(N)/s_b(N),\]
    then the parameters of the binomial are $N$ and
    \(
      \frac{(1+\beta(N)) x}{(1-x) + (1+\beta(N)) x}
    \)
    .
\end{description}

In the large population limit $N \rightarrow +\infty$, at long time scales
and in the regime of weak selection where $\beta(N) = \beta/N$, it is
well-known that the finite size model can be approximated by a solution of a stochastic
differential equation (namely a diffusion). 
This use of diffusion approximations in population genetics is now well
established. For an introduction to this subject, see \cite{Eth11},
\cite{EK86} and \cite{Ewe04}.

More precisely, 
define for all $N$ a continuous time process $(X^N)_{t\geq 0}$ by:
\[\forall t\in[0,1],  X^N_t = X^N_{\floor{t/N}}.\]
The diffusion approximation is the following:
\begin{thrm}
  [Wright--Fisher diffusion with selection]
  \label{thm:WFclassique}
  In the weak selection limit, the rescaled Wright--Fisher
  model $(X^N)_t$ converges weakly (in the Skorokhod sense) as $N\to \infty$  to the diffusion
  \( dX_t = \sqrt{a(x)} dB_t + b(X_t)dt\)
  generated by
  \( L = \frac{1}{2}a(x) \partial_{xx} + b(x)\partial_{x},\)
  where 
  \[\left\{
  \begin{aligned}
         a(x) &= x(1-x)\\
	 b(x) &=  \beta x(1-x).
  \end{aligned}
  \right.
\]
\end{thrm}
\begin{remark}
  If the white eggs survive better than the black ones, then $s_w(N) > s_b(N)$ so
  $\beta$ is positive; the diffusion drifts towards $x=1$. If black eggs are favored, 
  $\beta$ is negative and the drift is towards $0$. 
\end{remark}

\subsection{A Wright--Fisher model with indirect selection}
\label{sec:modelDefinition}
Let us now see how the basic model of Section~\ref{sec:basic_model} can be used
as a building block for a multiple generation model in the spirit of the
classical Wright--Fisher model, in order to study the evolution of the ``harmful''
trait along generations.

In the literature, various extensions of the Wright--Fisher model have been considered
under various scalings (see \cite{CS09} for a unifying approach). 
Frequency-dependent coefficients  may appear in such models but are often built-in
in the individual-based model (see \cite{CS14}). Another possible extension is 
to make the offsping play a (game-theoretic) game after the random mating step, 
see e.g.~\cite{Les05}. 
Frequency dependence appears more naturally in \cite{Gil74,Gil75} for modelling 
resistance to epidemics and comparing offspring distributions with different
variances; these papers do not however link the diffusion to a precise
individual-based model. For more details on these questions we refer to~\cite{Shp07,Tay09}
and references therein. 
In the more complicated setting of the evolution of continuous traits, 
several papers \cite{CFM06,CFM08} start with individual-based models and 
establish rigorously various limits, showing convergence to 
deterministic processes, SDEs or solutions of integro-differential equations. 
Finally in the literature a popular study concerns the fixation probabilities
(see \cite{Wax11} and \cite{MW07}) and problems arising at the boundaries.

Here we fix once and for all a sex-ratio by fixing the parameter $s>0$, and  supposing
that there are $s$ females for one male, \emph{i.e.} a proportion $1/(1+s)\in
(0,1)$ of the total population is male. Consider a large urn with
$n$ (male) balls,  let  $f_n = \floor{sn}$, and  define the state space
$\stateSpace = \left\{0,\frac{1}{n}, \frac{2}{n}, \ldots 1\right\}$: these are the possible
values for the proportion of white balls. 

We define an $\stateSpace$-valued Markov chain $(X^n_k)_{k\in\xN}$ as follows.
Suppose that the initial proportion of white balls at time $k=0$ in  the urn is
$X_0^n = x\in\stateSpace$: there are $w = xn$ white balls and $b = (1-x)n$
black balls.  The next state $X_1^n$ is chosen in two steps. 

\begin{description}
  \item[First step.]
The $f_n$ female pick partners according to the
single-generation model introduced previously: this leads to $\tXN_1$
reproduction with normal males and $\tYN_1$ reproduction with ``harmful''
males. As before, each of these reproductions creates a very large
number of ``eggs''. 
A proportion $s_w(N)$ (resp.\ $s_b(N)$) of white (resp.\ black) eggs survive, 
and the ratio $s_w(N)/s_b(N)$ is still denoted by $1 + \beta(N)$ with $\beta(N) = \beta/N$. 

After this step there is a very large number of eggs,  a proportion 
\begin{equation}
  \label{eq:defZTildeGeneral}
  \tZ^n_{1,\beta} = \frac{ (1+\beta(N)) \tXN_1}{ (1+\beta(N))\tXN_1 + \tYN_1}
\end{equation}
of which are white. 

\item[Second step.] Among all the eggs,  $n$  eggs are chosen uniformly 
  at random.  Once more, since the number of eggs is supposed to be very large,
the number  of white balls in the next generation follows a binomial law of
parameters $n$ and $\tZ^n_{1,\beta}$.
Finally divide
this number by $n$ to get $X^n_1$, the proportion of white balls at time $k=1$. 
\end{description}

We iterate the process to define $(X^n_k)_{k\geq 2}$.
As above we define a continuous process by accelerating time and let:
\[\forall t\geq 0,  X^n_t = X^n_{\floor{t/n}}.\]

Our main result  is a diffusion limit  for the rescaled process $(\frac{1}{n}X^n_k)_k$ 
with an explicit non-trivial drift towards $0$. The drift and volatility are expressed in terms 
of the following function: 
\begin{equation}
\begin{aligned}
  v_s: [0,1] & \to\xR, \\
           x &\mapsto v\left(\frac{x}{1+s},\frac{1-x}{1+s},\frac{s}{1+s}\right),
\end{aligned}
  \label{eq=defVs}
\end{equation}
where we recall that $v$ is defined by~\eqref{eq=defUv}. 
\begin{thrm}
  \label{thm:diffusion}
  If $s<1$, 
  the rescaled process $X^n_t$ converges weakly (in the Skorokhod sense) to the diffusion on $[0,1]$ given by the SDE:
  \( dX_t = \sqrt{a(x)} dB_t + b(X_t)dt\)
  and the corresponding generator
  \( L = \frac{1}{2}a(x) \partial_{xx} + b(x)\partial_{x},\)
  where 
  \[ \left\{
  \begin{aligned}
         a(x) &= \frac{x(1-x)}{v_s(x)},\\
	 b(x) &=  x(1-x)\PAR{\beta - \frac{v_s'(x)}{v_s^2(x)}} .
  \end{aligned}
  \right.
\]
\end{thrm}
\begin{remark}
If $s\geq 1$, we are only able to prove the convergence until the process reaches
$x = 1-y_0$; we currently do not know whether or not the behaviours at the boundary $x=1$
 differ for the discrete and continuous process. This possibly purely technical restriction
 prevents us from rigorously justifying the approximation of the discrete absorption probabilities
 and mean absorption time by their continuous counterparts, which is one of the usual 
 applications for diffusion approximations. 
\end{remark}
\begin{remark}
 The function $v_s$ is very nice, in particular it is strictly increasing ($v_{s}'>0$).
  If $s<1$, it is bounded away from zero. A statement with explicit bounds will be given
  below (Lemma~\ref{lem:propVs}). 
\end{remark}


A detailed study of the properties of this diffusion will be done in a forthcoming paper.
Let us just stress two points as regards the comparison with the classical 
model of Theorem~\ref{thm:WFclassique}:
\begin{enumerate}
	\item The variance is multiplied by $(1/v_s(x)) > 1$; this is a natural
	  consequence of the additional noise  in the first step. 
	  The precise factor may be heuristically justified 
	  as follows: 
	  when $n$ is large, there are $nv_{s}(x)+\mathcal{O}(1)$
	  successful reproductions, thus, with binomial resampling of
	  offspring, the male variance effective population size is also
	  $nv_{s}(x)+\mathcal{O}(1)$. Other examples of models with
	  frequency-dependent variance effective population sizes due to
	  polymorphism in life history traits can be found in \cite{Gil74},
	  \cite{Gil75}, \cite{Shp07} or \cite{Tay09}, where the strength of
	  selection on the alleles affecting the life history trait is also
	  inversely proportional to the census population size.
	\item To compare the drift coefficients, it is natural to consider the ``normalized'' 
	  quantity $2b=a$ which fully determines the scale functions and
	  hitting probabilities. In this light, up to a change of time, 
	  our modified diffusion corresponds to the classical 
	  one with a selection parameter $\beta(x) = \beta v_s(x) - \frac{v_s'(x)}{v_s(x)}$
	  that depends on $x$. If $s\to \infty$ this goes to $\beta$: all males
	  have a chance to reproduce and the harmful strategy has no effect.
	  If $\beta = 0$, $\beta(x)$ is negative
	  (and there is a non trivial drift towards $0$). In the general case, 
	  depending on the values of $\beta$ and $s$, there may be one or
	  more ``equilibrium'' points where the drift cancels out. These 
	  cases and their interpretation in biological terms will 
	  be studied in a forthcoming paper. 
\end{enumerate}

\paragraph{Outline of the paper.}
In Section~\ref{sec:singleBasic} we study the basic, single-generation model,
and prove Theorem~\ref{thm:avantage}; we also give concentration properties for
the number of reproductions. 
The asymptotic behaviour of $p_w$ and $p_b$ is studied in Section~\ref{sec:singleLimit} where
we prove Theorem~\ref{thm=approximation}. Finally, we prove the diffusion approximation for the multi-generation
model in Section~\ref{sec:diffusion}.

\section{The single-generation model --- basic properties}
\label{sec:singleBasic}
\subsection{The advantage of being harmful}
We consider here the simple model where we draw $f$ times from an urn with $w$ white balls
and $b$ black balls, where the white balls are removed when they are drawn and the black balls
are put back in the urn. 

Since similar-colored balls play the same role, $p_w$ is the probability of a successful reproduction
for a regular male, and $p_b$ the corresponding one for a ``harmful'' male. 
Finally let $q_w(w,b,f) = 1 - p_w(w,b,f)$ and $q_b(w,b,f) = 1 - p_b(w,b,f)$ be the 
probabilities of not being drawn. 

To prove Theorem~\ref{thm:avantage}, let us introduce a third function $q$ as follows.
  Add a single red ball to the $w$ white and $b$ black balls; let us call it Roger.
  Draw from the urn until the red ball is drawn or we have made $f$ draws; the white balls are not replaced 
  but the black ones are. Define:
  \begin{equation}
    \label{eq=defQr}
    q(w,b,f) = \prb{\text{Roger is not drawn}}.
  \end{equation}
  Since the color of a ball only matters if it is drawn, and only influences the 
  subsequent draws, it is easy to see that:
  \begin{align}
    \label{eq=pw_pb_and_q}
    q_w(w,b,f) &= q(w-1,b,f), &
    q_b(w,b,f) &= q(w,b-1,f).
  \end{align}
 Therefore it is enough to compare the probabilities that the red ball is never drawn, 
 when one ball goes from black to white. 

  We use a coupling proof. Suppose that the urn $1$ contains $w+b$ balls, numbered from $1$ 
  to $w+b$, where the first $w$ balls are white, the next $b-1$ are black and the last one is red. 
  Urn number $2$ is similar, except that the ball numbered $w$ is black instead of white. 
  Let $(U_i)$ be sequence of i.i.d.\ random numbers, uniformly distributed on $\{1, \ldots w+b\}$. 
  We define a joint evolution of the urns in the following way. 
  \begin{enumerate}
  	\item At the beginning of each step, look at the next random number; say its value is $k$. 
	\item 
	  \begin{itemize}
	\item If both balls numbered $k$ are still in their urns,  
	  choose these balls.
	\item If both balls numbered $k$ have been removed, try again with the next random number 
	  (this will only happen if the balls are both white). 
	\item If the ball numbered $k$ is still in one urn but has been removed from the other, 
	  then the ball that is present is chosen. Continue looking at the next random numbers
	  to choose a ball in the other urn. 
	  \end{itemize}
	\item At this point one ball is chosen in each urn. If any of the two is red, the process
	  is stopped in the corresponding urn. If a chosen ball is white it is removed from its urn. 
	\item Repeat until the two red balls have been chosen or $f$ draws have been made. 
  \end{enumerate}
  Each urn taken separately follows the initial process. Moreover, at any time, if the ball 
  numbered $i$ is still in the first urn, then it is also in the second one: indeed this is true 
  at the beginning, and if this is true at the beginning of a step it is true at the end of the step. 
  There are three possible situations:
  \begin{itemize}
  	\item both red balls are chosen at the same time; 
	\item the red ball is chosen in the first urn, but not in the second; 
	\item both red balls stay untouched during the $f$ steps. 
  \end{itemize}

  Therefore the probability that the red ball stays untouched is smaller in the first urn 
  than in the second urn, so 
  \[ 1 - p_b(w,b,f) = q_b(w,b,f) = q(w,b-1,f) \leq q(w-1,b,f) = q_w(w,b,f) = 1 - p_w(w,b,f).\]
  If $f$ is larger than $2$, and $w$ and $b$ are larger than $1$, the second case occurs with 
  positive probability so the inequality is strict. 
  This concludes the proof of Theorem~\ref{thm:avantage}. 

\subsection{Negative relation and concentration}
In this section we prove that the total number of reproductions $X$ and $Y$
defined at the beginning of the Introduction are sums of ``negatively related'' indicators; 
this implies very strong concentration bounds. 

In the original experiment, let us number the ``white'' males from $1$ to $w$, 
and the ``black'' males from $w+1$ to $w+b$. 
Let $B_i = \ind{\text{the $i$ male reproduces}}$. The total number of reproductions
is given by
\begin{align}
  X &= \sum_{i=1}^w B_i, 
  &
  Y &= \sum_{i=w+1}^b B_i.
  \label{eq=XandY}
\end{align}

The setting is quite close to the usual sampling from a bin with or without replacement, 
which leads to binomial and hypergeometric distributions. For these distributions, 
very strong approximation and concentration results can be proved using the fact that 
the indicators $B_i$ appearing in \eqref{eq=XandY} are ``negatively related'': intuitively, if 
a certain group of males have been chosen, the others are less likely to
be chosen. 
This approach is used in \cite{Jan94}, who refers 
to \cite{BHJ92} and \cite{JDP83} for further details on ``negatively
related/negatively associated'' variables.
\begin{dfn}
  [Negative relation, \cite{Jan94}]
  \label{df=negrel}%
  Let $I_1, \ldots I_k$ be indicator variables. If there
  exist indicator variables $J^{(i)}_j$ such that:
  \begin{itemize}
    \item $\forall j\neq i$, $J^{(i)}_j \leq I_j$, 
    \item for each $i$, the law of $(J^{(i)}_j)_j$ is the conditional law of $\mathbf{I}$ 
      given $I_i = 1$, 
  \end{itemize}
  then the variables $I_i$ are negatively related. If they are, then $(1 - I_i)_i$ are
  also negatively related. 
\end{dfn}
\begin{remark}
  For indicator variables, this corresponds to the existence of a ``decreasing size-biased coupling''
   in the terminology of \cite{Ross11}. However the boundedness condition used there
   to get concentration
   will not be satisfied with good constants.
\end{remark}
\begin{thrm}[Concentration for sums of negatively related indicators]
  \label{thm:concentration}
  Suppose that  $(I_i)$ are negatively related Bernoulli variables of parameter $p$. 
  Let $X = \sum_{i=1}^n I_i$, and let $X'$ be a binomial variable of parameters $n$ and $p$. 
  Then, for all $t\in\xR$, 
  \[ \esp{\exp(t X)} \leq \esp{\exp{tX'}}.\]
  Consequently, 
  \begin{align}
    \label{eq:momentOrdreTrois}
    \esp{\abs{X - \esp{X}}^3} &\leq 12 e n^{3/2} \\
    \label{eq:concentration}
    \prb{ \abs{X - \esp{X}} \geq D} &\leq \exp\PAR{ - \frac{D^2}{4n}}.
  \end{align}
\end{thrm}
\begin{proof}
  The key comparison of the Laplace transforms between $X$
  and the ``independent version'' $X'$ comes from \cite[Theorem~4]{Jan94}.
  Therefore any concentration bound obtained by the usual Chernoff trick for
  independent variables also holds when the indicators are negatively related. 

  The rest of the proof is routine and is included here for completeness. 
  Let $q = 1 - p$. For any $t$, 
  \begin{align*}
    \esp{\exp(t(X - \esp{X}))}
    &\leq \esp{\exp(t(X' - \esp{X'}))} \\
    &= \exp(-tnp) \PAR{ pe^t + q}^n \\
      &= \PAR{pe^{tq} + qe^{-tp}}^n.
  \end{align*}
  Here we use a small trick borrowed from \cite[p.~31]{GS01} and bound $e^x$ 
  by $x + e^{x^2}$, for $x = tq$ and $x = -tp$:
  \begin{align*}
    \esp{\exp(t(X - \esp{X}))}
      &\leq \PAR{pe^{t^2q^2} + qe^{t^2p^2}}^n \\
    &\leq \exp(t^2 n). 
  \end{align*}
  The deviation inequality \eqref{eq:concentration} follows by applying Markov's exponential
  inequality and choosing $t = D/2n$. For the moment bound, since $\frac{\abs{tx}^3}{3!}\leq 
  (\exp(tx) + \exp(-tx))$, 
  \begin{align*}
    \esp{ \abs{X - \esp{X}}^3} 
    &\leq \frac{6}{t^3} 2\exp(t^2 n)
  \end{align*}
  Choosing $t = n^{-1/2}$ yields \eqref{eq:momentOrdreTrois}. 
\end{proof}

\begin{thrm}
  Let $B_i$ be the indicator that the $i$\textsuperscript{th} male is chosen. The
  indicator variables $(B_i)_{i=1,\dots w+b}$ are negatively related. 
\end{thrm}
\begin{proof}
  Define $I_i = 1 - B_i$. By the remark in Definition~\ref{df=negrel}, it is
  enough to show that the $I_i$ are negatively related. 
  One may view the model as an urn occupancy problem: the $w+b$ balls become 
  urns, in which we put $f$ balls consecutively, not allowing more
  than one ball in each white urn; $I_i$ is the event ``the urn $i$ is empty
  at the end''. 
  For this type of problem, the property is
  standard and the $\cramped{J^{(i)}_j}$ may be defined
  explicitly in the following way.  First draw the balls and record the values
  of the $(I_i)$.  To define $\cramped{J^{(i)}_j}$, take all balls in the urn $i$ and
  reassign them to the other urns, following the same procedure.  
  Let $\cramped{J^{(i)}_j}$ be $1$ if the urn
  $j$ is empty after these reassignments.  The $\cramped{J^{(i)}_j}$ follow the 
  conditional distribution of $(I_1, \ldots I_{w=b}$ given $I_i = 1$,
  and since we only add balls to the urns $j$, $j\neq i$, $J^{(i)}_j \leq I_j$.  
\end{proof}

\section{The single-generation model --- large population limit}
\label{sec:singleLimit}
\subsection{Outline of the proof}

The goal of this section is to prove Theorem~\ref{thm=approximation} on the
convergence of $p_b$ and $p_w$ to a continuous function $v$. 
It will be slightly easier to work on the quantity $q(w,b,f)$ defined by
\eqref{eq=defQr}, and deduce the statements on $p_w$ and $p_b$ afterwards. 
This discrete function $q$ approximates the function $u$ defined by \eqref{eq=defUv}, 
and we also get convergence of the discrete differences of $q$ to the derivatives of $u$:
\begin{thrm}
  \label{thm:qConvergeVersU}
  For all $y_0 > 0$, 
  there exists $C(y_0)$ such that, for all $N$, 
   and all $ (w,b,f) \in \Omega_N(y_0)$,
  \begin{align}
    \label{eq:qConvergeVersU}
    \abs{(q - u^N)(w,b,f)} \leq \frac{C(y_0)}{N}, \\
    \label{eq:deltaXqConverge}
    \abs{(N\delta_x q - (\partial_x u)^N)(w,b,f)} \leq \frac{C(y_0)}{N}, \\
    \label{eq:deltaYqConverge}
    \abs{(N\delta_y q - (\partial_y u)^N)(w,b,f)} \leq \frac{C(y_0)}{N}. 
\end{align}
If $s<1$, there exists $C(s)$ such that the same bounds hold uniformly
on $\Omega_N(s)$, where $C(y_0)$ is replaced by $C(s)$ on the right hand side. 
\end{thrm}
The proof hinges on the following recurrence relation for $q$, which
follows by conditioning on the result of the first draw:
\begin{equation}
  \label{eq=reccQr}
q(w,b,f) =  \frac{w}{w+b+1} q(w-1,b,f-1) + \frac{b}{w+b+1} q( w, b,f - 1) 
\end{equation}
The main idea is then to view $q$ as a discrete version of $u$,
and the recurrence relation~\eqref{eq=reccQr} as an approximation 
of a relation between derivatives of $u$. 
The corresponding PDE for $u$ is derived in Section~\ref{sec:heuristic_PDE}, 
we show in Section~\ref{sec=solvePDE} that it is  explicitly solvable. 
Knowing this, we turn to the proof of Theorem~\ref{thm:qConvergeVersU} in the
following sections: the three convergences~\eqref{eq:qConvergeVersU}, \eqref{eq:deltaXqConverge}
and \eqref{eq:deltaYqConverge} are proved respectively in Sections~\ref{sec:convergence},%
~\ref{sec:convergence_deltaxq} and~\ref{sec:convergence_deltayq}. 
We show in Section~\ref{sec:from_q_to_pb} how to deduce the statements on $p_w$ and $p_b$ of 
Theorem~\ref{thm=approximation} from Theorem~\ref{thm:qConvergeVersU}. 
We conclude this long section by giving estimates in the same vein for second moments
in Section~\ref{sec:second_moments}.

\subsection{Identifying the limit function}
\label{sec:heuristic_PDE}

Let us now give a short heuristic argument for
finding the limit function $u$. Suppose that $u$ exists, and that 
all of the limits encountered below converge. 
Starting from the recurrence relation \eqref{eq=reccQr}, we introduce $q(w,b,f)$ on the 
right hand side, so that discrete differences appear:
\begin{align*}
q(w,b,f) &=  \frac{w}{w+b+1} q(w-1,b,f-1) + \frac{b}{w+b+1} q( w, b,f - 1)       \\
         &= \frac{w+b}{w+b+1} q(w,b,f) + \frac{w}{w+b+1} (q(w-1,b,f-1)-q(w,b,f)) \\
	 &\qquad + \frac{b}{w+b+1} (q( w, b,f - 1) - q(w,b,f)).
\end{align*}
Multiplying by $(w+b+1)$, we find after simplification:
\begin{align*}
q(w,b,f) &=  w (q(w-1,b,f-1)-q(w,b,f)) + b (q( w, b,f - 1) - q(w,b,f)) \\
         &= (w/N) \cdot N(q(w-1,b,f-1)-q(w,b,f))\\
         &+ (b/N)\cdot N(q( w, b,f - 1) - q(w,b,f)).
\end{align*}
Now if $N=w+b+f$ goes to $\infty$, and if $(w/N, b/N, f/N)$ converges to $(x,y,z)$, the
left hand side converges to $u$ and the right hand side 
to $x( -(\partial_x +\partial_z)u) + y (-\partial_zu)$, so that $u$ satisfies
\[ u + x\partial_x u + (x+y)\partial_z u = 0.\]
Since $q(w,b,0) = 1$, we also obtain $u(x,y,0) = 1$.
Summing up, if $q(w,b,f)$ converges ``in a good way'' to a function $u$, this
function satisfies an explicit first-order PDE on $\Omega$: 
\begin{equation}
  \label{eq=edp}
  \left\{\begin{aligned}
   &\forall (x,y,z)\in U, & u + F \cdot \nabla u &= 0, \\
   &\forall (x,y),        &             u(x,y,0) &= 1,
 \end{aligned}\right.
 \end{equation}
where $F$ is the vector field $F(x,y,z) = (x, 0, x+y)$. 

\begin{remark}
  \label{rem:jay} J.E.~Taylor\footnote{Personal communication.}
   suggested the following
  heuristic justification of the expressions of $T$, $u$ and $v$.
  Let $X_1$, \ldots $X_w$ be the number of reproduction
  attempts on each of the $w$ white balls, and $Y_1$, \ldots $Y_b$ be the
  number of attempts on the black balls. Since there is a total of $f$
  attempts, $\sum_{i=1}^w X_i + \sum_{j=1}^bY_j = f$ and in particular,
  $w\esp{X_1} + b\esp{Y_1} = f$.  Now we make two approximations. Firstly,
  $\prb{X_1=0} \approx \prb{Y_1=0} \approx u(x,y,z)$.  Secondly, the variable
  $Y_1$ should be approximately Poisson distributed: $Y_1$ counts successes in a large
  number of draws ($f$) that have a small chance of success. Then $\esp{X_1}
  \approx 1-u = v$, and the parameter $t$ of the distribution of $Y_1$
  satisfies $e^{-t} = u$, so $t=-\log(1-v)$. Inserting this in the equation on
  expectations yields $x v -y\log(1-v) = z$, which is another form of the
  equations~\eqref{eq=defT} and~\eqref{eq=defUv} defining~$v$. 

  A complete justification of these arguments, and in particular of the Poisson
  approximation, should be possible but could be quite involved, since the dependence
  between the draws is not easy to take into account. 
\end{remark}

\subsection{Resolution of the PDE}
\label{sec=solvePDE}
This first order PDE \eqref{eq=edp} can be solved by the method of characteristics. 
  We look for trajectories $M(t) = ( x(t) ; y(t) ;  z(t))$
    that satisfy the characteristic equation:
 \[ \frac{d}{dt}M(t) =  - F(M(t))\]
 The solution is:
 \[
 \begin{cases}
   x(t) = x_0e^{-t} \\
   y(t) = y_0 \\
   z(t) = x_0 (e^{-t} - 1) - y_0t + z_0. 
 \end{cases}
 \]

 Now  $h(t) = u(M(t))$ satisfies:
 \[ \frac{dh}{dt} = \nabla u \cdot \frac{d}{dt}M(t) = -\nabla_u(M(t)) \cdot F(M(t))  = u(M(t)) = h(t). \]
 Therefore $h(t) = h(0) \exp(t).$ 
 Suppose $T=T(x_0,y_0,z_0)$ is a solution of \eqref{eq=defT}, i.e.{} $z(T) = 0$. Then $h(T) = u(M(T)) = 1$
 thanks to the boundary condition.
 Finally~: 
 \begin{equation}
   \label{eq=defU}
   u(x_0,y_0,z_0) = h(0) =  u(M(T)) \exp(-T) = \exp(-T(x_0,y_0,z_0)).
 \end{equation}
 
 \begin{thrm}
   [Properties of the solution]
   \label{thm:propertiesOfU}
   If $(x,y,z)\in\Omega$ and if $y>0$, the equation \eqref{eq=defT} defining
   $T$ has a unique solution. 
   The function $u$ defined by \eqref{eq=defU} is smooth on the interior
   domain $\{(x,y,z)\in (\xR_+^\star)^3, x+y+z<1\}$. For any $y_0>0$, 
   there exists a constant $C(y_0)$ such that
     for all $(x,y,z)\in\Omega(y_0)$, and all   $(i,j)$,
     \begin{align}
     \abs{u(x,y,z)}           &\leq C(y_0), &
     \abs{\partial_iu(x,y,z)} &\leq C(y_0), &
     \abs{\partial_i\partial_j u (x,y,z)}& \leq C(y_0). 
   \end{align}
   If $s<1$, similar bounds hold uniformly on $\Omega(s)$. 
 \end{thrm}
 \begin{proof}
   If $y$ is strictly positive, $\phi:t\mapsto x(e^{-t} - 1) - y t + z$ is 
   a strictly decreasing smooth function such that $\phi(0) = z$ and $\phi(z/y) < 0$. 
   Therefore $T$ is unique and depends smoothly on $x,y,z$ by the implicit function theorem. 
   Its derivatives are given by:
   \begin{align*}
     \partial_x T &= \frac{e^{-T} - 1}{xe^{-T} + y}; &
     \partial_y T &= \frac{-T}{xe^{-T} + y}         ;&
     \partial_z T &= \frac{1}{xe^{-T} + y}.
   \end{align*}
   On $\Omega(y_0)$,  $T$ is positive and smaller than $1/y_0$, therefore these quantities are bounded.
   The same is true for the higher order derivatives. 

   If $s<1$, recall that on $\Omega(s)$, 
   \begin{align}
     \label{eq:OnOmegaS}
     z&\leq s(x+y)  &
     x-z &\geq (1-s)/(2 + 2s).
   \end{align}
   By the first condition, we obtain $\phi( \ln (1/1-s)) \leq y( \ln(1-s)+s) \leq 0$ which implies that $T\leq \ln(1/(1-s))$. 
   Together with the second condition, this implies that
   the denominator $xe^{-T} + y \geq xe^{-T} \geq x(1-s) \geq (1-s)^2/(2+2s)$. 
   This proves the claimed bounds. 
 \end{proof}

\subsection{Convergence}
\label{sec:convergence}
In this section we  prove~\eqref{eq:qConvergeVersU}.
Let $u$ be the solution~\eqref{eq=defU} of the continuous PDE, and $u^N$ its
 discretization defined by~\eqref{eq:defDiscretization}.
 If the recurrence relation~\eqref{eq=reccQr} can be seen
 as a numerical scheme for the resolution of the PDE~\eqref{eq=edp}, 
 $u^N$ should approximately satisfy~\eqref{eq=reccQr}.
 Define $R_N$ to be the corresponding difference:
 \begin{equation}
   \label{eq=defRnI}
   R_N(w,b,f) = u^N(w,b,f) - \frac{w}{w+b+1}u^N(w-1,b,f-1) - \frac{b}{w+b+1}u^N(w,b,f-1).
 \end{equation}
 \begin{proposition}
   For all $y_0>0$, there exists $C(y_0)$ such that for all $N$, 
   \[
   \forall (w,b,f)\in \Omega_N(y_0), \quad 
   \abs{R_N(w,b,f)} \leq \frac{C(y_0)}{N^2}.
   \]
   If $s<1$, a similar bound holds uniformly on $\Omega_N(s)$. 
 \end{proposition}
 \begin{proof}
   Let $m_N(w,b,f)$ be the sup of the second derivatives of $u$ on the cell
   $[(w\pm 1)/N]\times [(b \pm1)/N]\times[ (f\pm 1)/N]$.
 Let $\propNorm = (w/N,b/N,f/N)$, so that $u^N(w,b,f) = u(\propNorm)$. 
 Multiply \eqref{eq=defRnI} by $(w+b+1)$ and apply Taylor's formula:
   \begin{align*}
     (w+b+1)R_N(w,b,f) 
     &= (w+b+1)u(\propNorm)
     - w\left(
         u(\propNorm) -\frac{1}{N}\partial_x u(\propNorm) - \frac{1}{N}\partial_zu(\propNorm)
       \right) \\
     &\quad
     - b \left(
         u(\propNorm) - \frac{1}{N} \partial_zu(\propNorm)
       \right)
     + (w+b+1)\epsilon(w,b,f)
   \end{align*}
   where $\abs{\epsilon(w,b,f)} \leq \frac{1}{N^2} m_N(w,b,f)$.
   So
   \begin{align*}
     (w+b+1)R_N(w,b,f) 
     &= u(\propNorm) + w\PAR{\frac{1}{N}\partial_xu(\propNorm) + \frac{1}{N}\partial_zu(\propNorm)}
        + b\frac{1}{N} \partial_zu(\propNorm) \\
     &\qquad + (w+b+1)\epsilon(w,b,f).
   \end{align*}
   Since $u$ solves the PDE, all terms vanish except the last one, so
   \[ \abs{R_N(w,b,f)} \leq \frac{m_N(w,b,f)}{N^2}.\]
   The controls on the derivatives of $u$ given by
   Theorem~\ref{thm:propertiesOfU} show that $m_N$ is bounded by some $C(y_0)$ on 
   $\Omega_N(y_0)$, and by some $C(s)$ on $\Omega_N(s)$: 
   this concludes the proof.
 \end{proof}

 Now let $e_N(w,b,f)$ be the difference  $q(w,b,f) - u^N(w,b,f)$. 
 By the recurrence relation~\eqref{eq=reccQr} and the definition~\eqref{eq=defRnI} of $R_N$, 
 for $w\geq 1$ and $f\geq 1$ we get:
 \[
    e_N(w,b,f) = \frac{w}{w+b+1}e_N(w-1,b,f-1) + \frac{b}{w+b+1}e_N(w,b,f-1) - R_N(w,b,f).
 \]
 This still holds for $w=0$ if we define $e_N(-1,b,f) = 0$. \\
 Now define $\overline{e_N}(f) = \max \{ \abs{e_N(w,b,f)} : (w,b)\in \xN^2, (w,b,f) \in \Omega_N(y_0)\}$. 
 The key fact is that, if $(w,b,f)$ is in $\Omega_N(y_0)$, the same is true for $(w-1,b,f-1)$ and $(w,b,f-1)$. 
 Therefore:
 \begin{align*}
   \overline{e_N}(f) 
   &\leq \overline{e_N}(f-1) + \max\{R_N(w,b,f): w,b ; (w,b,f)\in \Omega_N(y_0)\} \\
   &\leq \overline{e_N}(f-1) + \frac{C(y_0)}{N^2}.
 \end{align*}
 By induction, since $f\leq N$, 
 \[
   \overline{e_N}(f) \leq \overline{e_N}(0) + \frac{C(y_0)}{N}.
 \]
 Since $e_N(w,b,0) = 0$, we are done. 

 To prove the bounds on $\Omega(s)$, the strategy is exactly the same. Once more, 
 the crucial step is to remark that $(w-1,b,f-1)$ and $(w,b,f-1)$ belong to 
 $\Omega_N(s)$ whenever $(w,b,f)\in\Omega_N(s)$: this stability is the reason 
 behind the very definition of $\Omega(s)$. 

 \subsection{Derivative in the  \texorpdfstring{$x$}{x} direction}
 \label{sec:convergence_deltaxq}
 Let us now prove the convergence of the (renormalized) finite differences of $q$ 
 to the derivatives of $u$. We proceed in three steps:
 \begin{enumerate}
 	\item find a recurrence relation for the finite differences;
	\item find a PDE for the derivative;
	\item use the PDE to show that the discretization of the derivatives
	  almost follows the same recurrence relation as the finite differences.
 \end{enumerate}
 We begin by the convergence of the derivatives in the $x$ direction. 
 Define $u_x = \partial_xu$, and
 recall that $\delta_x q(w,b,f)$ is the finite difference:
 \[\delta_x q(w,b,f) = q(w+1,b,f) - q(w,b,f).\]
 \paragraph{Step 1.} In order to obtain a recurrence relation for $\delta_x q$, starting
 from its definition, we apply~\eqref{eq=reccQr} two times to $q(w+1,b,f)$ and $q(w,b,f)$:
 \begin{align}
   \delta_x q(w,b,f) 
   &= \frac{w+1}{w+b+2} q(w,b,f-1) + \frac{b}{w+b+2}q(w+1,b,f-1) 
                                   \notag \\ & \qquad
      - \frac{w}{w+b+1} q(w-1,b,f-1) - \frac{b}{w+b+1}q(w,b,f-1)
                                   \notag\\
   &= \frac{w}{w+b+1}\delta_xq(w-1,b,f-1) + \frac{b}{w+b+2}\delta_xq(w,b,f-1)
                                    \notag\\ &\qquad
     + \left( \frac{w+1}{w+b+2} - \frac{w}{w+b+1}\right) q(w,b,f-1) 
                                    \notag\\ &\qquad 
     + \left( \frac{b}{w+b+2} - \frac{b}{w+b+1}\right) q(w,b,f-1)
                                    \notag\\
   &= \frac{w}{w+b+1}\delta_xq(w-1,b,f-1) + \frac{b}{w+b+2}\delta_xq(w,b,f-1) 
                                    \notag\\ &\qquad
      \label{eq=reccDeltaXQ}
      + \frac{1}{(w+b+1)(w+b+2)}  q(w,b,f-1).
 \end{align}

 \paragraph{Step 2.} Now let us find a PDE for $u_x$. 
 Recall  that $u_x = \partial_x u$. Since $u+ x\partial_xu + (x+y)\partial_zu =0$, 
 $u_x$ satisfies:
 \[ 2u_x + x\partial_x(u_x) + \partial_zu + (x+y)\partial_zu_x = 0.\]
 Plugging the first equation into the second gives:
 \[ 2u_x + x\partial_x u_x  - \frac{1}{x+y}u -\frac{x}{x+y}u_x + (x+y)\partial_z u_x = 0,\]
 which simplifies to:
 \begin{equation}
   \label{eq=edpDxu}
   \frac{x+2y}{x+y} u_x + x\partial_x u_x  + (x+y)\partial_z u_x = \frac{1}{x+y}u.
 \end{equation}

 \paragraph{Step 3.} Let $u_x^N$ be the discretization of $u_x$.
 This function should approximately satisfy the same relation 
 as $N\delta_xq$, i.e.{} the product of Equation~\eqref{eq=reccDeltaXQ} by $N$. 
 Denote by $R_N$ the error in this approximation, i.e. 
 $R_N$ is such that:
 \begin{align}
 u_x^N(w,b,f) &= \frac{w}{w+b+1} u_x^N(w-1,b,f-1) + \frac{b}{w+b+2}u_x^N(w,b,f-1)
 \notag \\
 &\qquad + \frac{N}{(w+b+1)(w+b+2)} q(w,b,f-1)
 + R_N(w,b,f). 
 \label{eq=defRn}
 \end{align}
 The error $e_N = u_x^N - N\delta_xq$ satisfies:
 \[
 e_N(w,b,f) = \frac{w}{w+b+1} e_N(w-1,b,f-1) + \frac{b}{w+b+1}e_N(w,b,f-1) + R_N(w,b,f)
 \]
 so the same proof as before applies, provided we show that
 \begin{itemize}
   \item $R_N$ is $\mathcal{O}(N^{-2})$;
   \item $e_N(w,b,0)$ is small. 
 \end{itemize}

 \begin{lmm}[$R_N$ is small]
   For any $y_0$, there exists a $C(y_0)$ such that
 \begin{equation}
   \label{eq=boundOnRn}
 \forall (w,b,f)\in \Omega_N(y_0),  \quad R_N(w,b,f) \leq \frac{C(y_0)}{N^2}.
 \end{equation}
 The same holds uniformly on $\Omega_N(s)$ if $s<1$. 
 \end{lmm}
 \begin{proof}
   Multiply \eqref{eq=defRn} by $(w+b+1)$ and use Taylor's formula:
   \begin{align*}
   &(w+b+1)u_x^N(w,b,f) \\
   &= w\left(u_x^N(w,b,f) - \frac{1}{N}(\partial_xu_x)^N(w,b,f) 
                          - \frac{1}{N}(\partial_zu_x)^N(w,b,f)\right) \\
   &\qquad  + b\left(1-\frac{1}{w+b+2}\right)
               \left(u_x^N(w,b,f) - \frac{1}{N} (\partial_zu_x)^N(w,b,f)\right) \\
   &\qquad + \frac{N}{w+b+2} q(w,b,f-1) + (w+b+1)(R_N(w,b,f) + \epsilon(w,b,f)),
 \end{align*}
 where $\abs{\epsilon(w,b,f)} \leq m_N(w,b,f)$. 
 Gather all the $u_x^N$ terms on the left hand side.
   \begin{align*}
     \frac{w+2b+2}{w+b+2}u_x^N(w,b,f)
   &= -w\left( \frac{1}{N}(\partial_xu_x)^N(w,b,f) + \frac{1}{N}(\partial_zu_x)^N(w,b,f)\right) \\
   &\quad  - \frac{b}{N}\left(1-\frac{1}{w+b+2}\right) \partial_zu_x^N(w,b,f) \\
   &\quad + \frac{N}{w+b+2} q(w,b,f-1) + (w+b+1)(R_N(w,b,f) + \epsilon(w,b,f)).
 \end{align*}
 To use the fact that $u_x$ satisfies \eqref{eq=edpDxu} we isolate the relevant terms:
   \begin{align*}
   &\frac{w+2b}{w+b}u_x^N(w,b,f) - \frac{2b}{(w+b)(w+b+2)}u_x^N(w,b,f) \\
   &\quad = -w\left(
                \frac{1}{N}(\partial_xu_x)^N(w,b,f) + \frac{1}{N}(\partial_zu_x)^N(w,b,f)
	      \right) \\
   &\qquad  - \frac{b}{N}  (\partial_zu_x)^N(w,b,f) + \frac{b}{N}\frac{1}{w+b+2}(\partial_zu_x)^N(w,b,f)\\
   &\qquad +\frac{N}{w+b}u^N(w,b,f) + \frac{N}{w+b} (q(w,b,f-1)-u^N(w,b,f)) \\
   &\qquad - \frac{2N q(w,b,f-1)}{(w+b)(w+b+2)} + (w+b+1)(R_N(w,b,f) + \epsilon(w,b,f)).
 \end{align*}
 Thanks to \eqref{eq=edpDxu} applied at the point $(w/N,b/N,f/N)$, we obtain:
   \begin{align*}
   &- \frac{2b}{(w+b)(w+b+2)}u_x^N(w,b,f) \\
   & = \frac{b}{N}\frac{1}{w+b+2}(\partial_zu_x)^N(w,b,f) \\
   &\quad + \frac{N}{w+b} (q(w,b,f-1)-u^N(w,b,f)) - \frac{2N q(w,b,f-1)}{(w+b)(w+b+2)} \\
   &\quad + (w+b+1)(R_N(w,b,f) + \epsilon(w,b,f)).
 \end{align*}
 Isolating $R_N$ in this equation and using the fact that $b\geq Ny_0$,
 the bounds $\abs{u^N_x} \leq 1$, $\abs{\partial_z u_x^N}\leq C(y_0)$ as stated in Theorem \ref{thm:propertiesOfU}, and
 the approximation result on $q$ (Equation~\eqref{eq:qConvergeVersU}), we get
 the bound~\eqref{eq=boundOnRn}. On $\Omega(s)$ the proof is the same, replacing
 the lower bound on $b$ on the denominator by the control
 \[ w = Nx \geq N\frac{1-s}{2+2s}.\]
 This concludes the proof of the lemma. 
 \end{proof}

 To conclude the proof of \eqref{eq:deltaXqConverge}, 
 we need only consider the base case $f = 0$.  Since $q(\cdot,\cdot,0)$ is identically $1$,
$\delta_xq$ is zero for $f=0$. Similarly $u$ is identically $1$ so its 
$x$-derivative is $0$, so $e_N(w,b,0) = 0$, and \eqref{eq:deltaXqConverge}
follows by the same induction as before. 

 \subsection{The other derivatives}
 \label{sec:convergence_deltayq}
 Let us now turn to the convergence of the $y$-derivative $u_y$. 
 This function satisfies the PDE:
 \begin{equation}
   \label{eq=edpDyu}
   u_y + x \partial_xu_y + (x+y)\partial_z u_y = -\frac{1}{x+y}u - \frac{x}{x+y}u_x.
 \end{equation}
 Note that the right hand side depends on $u$ and $u_x$, for which we have already proved
 approximation results. 

 Recall that $\delta_yq(w,b,f) = q(w,b+1,f) - q(w,b,f)$. Using the recurrence
 relation \eqref{eq=reccQr} for $q$ we find first that
 \begin{align*}
   &\delta_yq(w,b,f) \\
   &= \frac{w}{w+b+2}q(w-1,b+1,f-1) + \frac{b+1}{w+b+2}q(w,b+1,f-1) \\
   &\qquad 
      - \frac{w}{w+b+1}q(w-1,b,f-1) - \frac{b}{w+b+1}q(w,b+1,f-1)
\\
   &= \frac{w}{w+b+1}\delta_yq(w-1,b,f-1) + \frac{b}{w+b+1}\delta_yq(w,b,f-1) \\
   &\qquad 
      - \frac{w}{(w+b+1)(w+b+2)}q(w-1,b+1,f-1) \\
   &\qquad
      + \frac{w+1}{(w+b+1)(w+b+2)}q(w,b+1,f-1)
\\
   &= \frac{w}{w+b+1}\delta_yq(w-1,b,f-1) + \frac{b}{w+b+1}\delta_yq(w,b,f-1) \\
   &\qquad 
      + \frac{w}{(w+b+1)(w+b+2)}\delta_xq(w-1,b+1,f-1) \\
   &\qquad 
   + \frac{1}{(w+b+1)(w+b+2)}q(w,b+1,f-1).
\end{align*}
Once more, the discretization $u_y^N$ of $u_y$ should behave approximately like 
$N\delta_y q$. Define $R_N$ to be the error in this approximation, 
i.e.\ $R_N$ is such that:
\begin{align*}
  u_y^N(w,b,f) &= \frac{w}{w+b+1}u_y^N(w-1,b,f-1) + \frac{b}{w+b+1}u_y^N(w,b,f-1) \\
   & + \frac{wN}{(w+b+1)(w+b+2)}\delta_xq(w-1,b+1,f-1)\\
   & + \frac{N}{(w+b+1)(w+b+2)}q(w,b+1,f-1) \\
   & + R_N(w,b,f). 
\end{align*}
To study $R_N$, multiply both sides by $(w+b+1)$, and use Taylor's formula:
\begin{align*}
 &(w+b+1) u_y^N(w,b,f) \\
   &= w u_y^N(w,b,f) + b u_y^N(w,b,f) 
   - \frac{w}{N}( (\partial_x + \partial_z)u_y)^N(w,b,f) 
   - \frac{b}{N}(  \partial_zu_y)^N(w,b,f ) \\ 
   &\qquad 
      + \frac{wN}{w+b+2}\delta_xq(w-1,b,f-1) + \frac{N}{w+b+2}q(w,b+1,f-1) \\
   &\qquad + (w+b+1) R_N(w,b,f) + \epsilon(w,b,f),
\end{align*}
where $\epsilon = \bigO(N^{-1})$ (uniformly on $\Omega(y_0)$ and on
$\Omega(s)$). The term $(w+b)(u_y)^N$ cancels out. The remaining terms almost
cancel out thanks to \eqref{eq=edpDyu}, and we are left with
\[
  \begin{split}
  (w+b+1) R_N(w,b,f)
  = \bigO(1/N) + \left(\frac{N}{w+b+2}q(w,b+1,f-1) - \frac{N}{w+b}u^N(w,b,f)\right) \\
  + \left( \frac{wN}{w+b+2} \delta_xq(w-1,b+1,f-1) - \frac{w}{w+b}u_x^N(w,b,f)\right).
\end{split}
\]
Using the approximation results \eqref{eq:qConvergeVersU}, \eqref{eq:deltaXqConverge}, and  the fact that $b\geq Ny_0$
on $\Omega_N(y_0)$, or that $w\geq N \frac{1-s}{2+2s}$ on $\Omega_N(s)$, 
we can prove that
\[ \abs{R_N(w,b,f)} \leq C(y_0)N^{-2}.\]
The last step is the same as before: the difference $e_N = N\delta_yq - u_y^N$ 
satisfies the nice recurrence relation
\begin{align*}
  e_N(w,b,f) = \frac{w}{w+b+1}e_N(w-1,b,f-1) + \frac{b}{w+b+1}e_N(w,b,f-1) - R_N(w,b,f).
\end{align*}
For the base case $(f =0)$, $e_N$ is identically zero, and we get by induction:
\[ 
  \max\left\{ \abs{e_N(w,b,f)} :  (w,b) \text{ such that } (w,b,f)\in\Omega_N(y_0) \right\}
  \leq \frac{C(y_0)f}{N^2},
\]
which proves \eqref{eq:deltaYqConverge}  since $f\leq N$. 
The proof is similar on $\Omega_N(s)$.

\subsection{Proof of Theorem~\ref{thm=approximation}}
\label{sec:from_q_to_pb}
Let us see how the statements of Theorem~\ref{thm=approximation} regarding $p_b$ and $p_w$ 
may be deduced from Theorem~\ref{thm:qConvergeVersU}. The three proofs being similar, 
we only consider the last equation~\eqref{eq=cvg_differences},
that is, we prove
  \begin{align*}
    \forall (w,b,f) \in \Omega_N(y_0), \quad
    \abs{ p_b(w,b,f) - p_w(w,b,f) - \frac{1}{N}(\partial_xv - \partial_yv)^N(w,b,f)} \leq \frac{C(y_0)}{N^2}.
  \end{align*}

Recall that $p_w$ and $p_b$ can be expressed in terms of $q$~: by \eqref{eq=pw_pb_and_q}, 
  \begin{align}
    p_w(w,b,f) &= 1 - q_w(w,b,f) = 1 - q(w-1,b,f),  \\
    p_b(w,b,f) &= 1 - q_b(w,b,f) = 1 - q(w,b-1,f).
  \end{align}
First write everything in terms of $q$ and $u$, recalling that $u = 1-v$:
  \begin{align}
    \notag
     &p_b(w,b,f) - p_w(w,b,f) - \frac{1}{N}(\partial_xv - \partial_yv)^N(w,b,f) \\
     \notag
     &\qquad = q(w-1,b,f) - q(w,b-1,f) + \frac{1}{N}(\partial_xu - \partial_yu)^N(w,b,f) \\
     \label{eq=jackestici}
     &\qquad = q(w-1,b,f) - q(w,b,f) + \frac{1}{N}(\partial_x u)^N(w,b,f) \\
     \label{eq=joeestla}
     &\qquad\quad + q(w,b,f) - q(w,b-1,f) - \frac{1}{N}(\partial_y u)^N(w,b,f).
  \end{align}
  The absolute value of the last line~\eqref{eq=joeestla} satifies:
  \[\begin{split}
  \abs{q(w,b,f) - q(w,b-1,f) - \frac{1}{N} (\partial_y u)^N(w,b,f)} \\
  \ \ \leq\frac{1}{N} \abs{N\delta_yq(w,b-1,f) - (\partial_y u)^N(w,b-1,f)} \\
  + \frac{1}{N} \abs{(\partial_yu)^N(w,b-1,f) - (\partial_y u)^N(w,b,f)}.
\end{split}
\]
Fix $0<y'_0<y_0$ to ensure that $(w,b-1,f)$ is in $\Omega_N(y'_0)$ when
$(w,b,f)$ is in $\Omega_N(y'_0)$.  By~\eqref{eq:deltaYqConverge} the first term
is bounded by $\frac{C(y'_0)}{N^2}$.  The controls on $u$ from
Theorem~\ref{thm:propertiesOfU} imply that the second term is also bounded by
$C(y'_0)/N^2$. The same arguments may be applied to the terms
in~\eqref{eq=jackestici}. This concludes the proof of~\eqref{eq=cvg_differences}.

\subsection{Second moments}
\label{sec:second_moments}
In order to derive the diffusion limit, we will need information on the
covariance structure of the couple $(X,Y)$. This information
will be deduced in Section~\ref{sec:aboutXTilde} from 
estimates on the following variant of the function $q$. 

\begin{dfn}
  \label{dfn:qTilde}
For any positive integers $w$, $b$ and $f$, 
we denote by $\tq(w,b,f)$ the probability
that in an urn composed of $w$ white balls, $b$ black balls and $2$ red balls, the $2$ red balls are not drawn after $f$ trials.
\end{dfn}

By conditioning, we see that  $\tq$ satisfies the recurrence relation:
\[
  \tq(w,b,f) = \frac{w}{w+b+2} \tq(w-1,b,f-1) + \frac{b}{w+b+2}\tq(w,b,f-1).
\]
As before we prove that $\tq$ converges in some sense to 
a limit function $\tu$. The same heuristic reasoning as before leads to
 the candidate PDE:
\[ -\frac{2}{x+y} \tu -  \frac{x}{x+y}(\partial_x \tu + \partial_z\tu)  - \frac{y}{x+y}\partial_z \tu = 0\]
which we rewrite as
\[ 2\tu + x\partial_x \tu + (x+y)\partial_z \tu  = 0,\]
with the boundary condition
\[\tu(x,y,0) = 1.\]
The only difference between this equation and the PDE~\eqref{eq=edp}
is that $F$ is replaced by $F/2$. We solve this new equation in the same way. The characteristics are:
 \[
 \begin{cases}
   x(t) = x_0e^{-t/2} \\
   y(t) = y_0 \\
   z(t) = x_0 (e^{-t/2} - 1) - y_0t/2 + z_0. 
 \end{cases}
 \]
 The $\tT$ that satisfies $z(\tT) = 0$ is just $\tT = 2T$, so the solution
 $\tu$ is given by:
  \[
    \tu(x_0,y_0,z_0) = \exp(2T(x_0,y_0,z_0)) = u^2(x_0,y_0,z_0).
  \]
 Following the same strategy as before, 
 we prove:
 \begin{proposition}[Asymptotics of $\tq$]
   For any $y_0$ there exists $C(y_0)$ such that for all $N$ and
   for all $(w,b,f)\in \Omega_N(y_0)$, 
   \label{prop:qTilde}
 \begin{align*}
   \abs{(\tq -  (u^N)^2)(w,b,f)} \leq \frac{C(y_0)}{N},   \\
   \abs{( \tq(w,b-1,f) - \tq(w-1,b,f) ) - \frac{2}{N}( u(\partial_x - \partial_y)u)^N(w,b,f)} 
   \leq \frac{C(y_0)}{N}.
 \end{align*}
 Similar bounds hold uniformly on $\Omega_N(s)$ if $s<1$. 
 \end{proposition}

\section{The multi-generation model}
\label{sec:diffusion}
\subsection{Main line of the proof}
To prove the diffusion limit stated in Theorem~\ref{thm:diffusion}, we 
follow the presentation of Durrett in~\cite{Dur96}. 
For each $n$, we have defined a Markov chain $(X^n_k)_{k\in\xN}$, that 
lives on the state space $\stateSpace = \{0, \frac{1}{n}, \ldots, 1\}\subset \xR$. 
Let $\espX{\cdot}$ and $\varX{\cdot}$  denote the expectation and 
variance operators for the Markov chained started at $X^n_0 = x$.
Define, for each $n$ and each $x\in\stateSpace$, the ``infinitesimal
variance'' $a^n(x)$ and the ``infinitesimal mean'' $b^n(x)$ by:
\begin{align}
  \label{eq:defInfVar}
  a^n(x) &=  n \varX{X^n_1}, \\
  \label{eq:defInfMean}
  b^n(x) &= n\left( \espX{X^n_1} - x\right),
\end{align}
and let 
\[
  c^n(x) = n\espX{ \abs{X^n_1 - x}^3}.
\]
  Suppose additionally that $a$ and $b$
  are two continuous functions for which the
  martingale problem is well posed, i.e., for each $x$ there 
  is a unique measure $P_x$ on $\mathcal{C}([0,\infty),\mathbb{R})$
  such that $P_x[X_0 = x] = 1$ and 
  \[ X_t - \int_0^t b(X_s) ds \qquad \text{ and } \qquad X_t^2 - \int_0^t a(X_s)ds \]
  are local martingales.  In this setting, the convergence 
  of the discrete process to its limit is a consequence of the following result. 

  \begin{thrm}[Diffusion limit,\cite{Dur96} Theorem 8.7.1 and Lemma~8.8.2]
  \label{thm:Durrett}
  Suppose that the following three conditions hold.
  \begin{enumerate}
  	\item 
  The infinitesimal mean and variance converge uniformly:
  \begin{align*}
    \lim_n \sup_{x\in\stateSpace} \abs{a^n(x) - a(x)} = 0, \\
    \lim_n \sup_{x\in\stateSpace} \abs{b^n(x) - b(x)} = 0.
  \end{align*}
\item  The size of the discrete jumps is small enough:
  \[
    \lim_n \sup_{x\in\stateSpace} c^n(x) = 0.
  \]
\item The initial condition $X_0^n = x^n$ converges to $x$. 
  \end{enumerate}
  Then the renormalized process converges to the diffusion $X_t$. 
\end{thrm}
\begin{remark}
  The original formulation is $d$-dimensional and considers
  diffusions on the whole space, therefore it includes additional details
   that will not be needed here. 
\end{remark}
Using this result, Theorem~\ref{thm:diffusion} will 
follow once we prove that the martingale problem 
is well posed and we show the following estimates.

\begin{proposition}[Infinitesimal mean and variance]
  \label{prop:infMeanInfVar}
  The following estimates hold:
  \begin{align}
    \label{eq:infVar}
    a_n(x) &=  \frac{x(1-x)}{v_s(x)} + \bigO(1/\sqrt{n}),\\
    \label{eq:infMean}
    b_n(x) &= x(1-x) \PAR{\beta -  \frac{v'_s(x)}{v_s^2(x)}} + \bigO(1/\sqrt{n}),\\
    \label{eq:noJumps}
    c_n(x) &= \bigO(1/\sqrt{n}),
  \end{align}
  where the ``$\bigO$'' holds:
  \begin{itemize}
    \item uniformly on $\stateSpace\cap [0,x_0]$, for all $x_0<1$, if $s\geq 1$, 
    \item uniformly on the entire space $\stateSpace$, if $s<1$. 
  \end{itemize}
\end{proposition}

\paragraph*{Outline of the section.}
The remainder of this section is organized as follows. We verify in
Section~\ref{sec:well_posed} that the martingale problem is well posed. In
Section~\ref{sec:aboutXTilde}
we use the convergence results of the single generation model to study the
``first step'' and get information on the asymptotics of the random number of
reproductions. Focusing first on the case where $\beta = 0$ (i.e.\ there
is no ``direct'' fitness advantage), 
we  prove the formula~\eqref{eq:infMean} for the infinitesimal mean
in Section~\ref{sec:infMean}, postponing an estimate of a remainder term to
Section~\ref{sec:remainder}. 
The infinitesimal variance formula~\eqref{eq:infVar} and the
control~\eqref{eq:noJumps} on the higher moments of the jumps
are proved in sections~\ref{sec:infVariance}
and~\ref{sec:tightness}, still in the case $\beta = 0$. 
Finally we show in Section~\ref{sec:betaNonZero} how to 
recover all these results in the case where $\beta$ is arbitrary.

\subsection{The martingale problem is well posed}
\label{sec:well_posed}
Recall the definition \eqref{eq=defVs} of the function $v_s$: 
 \begin{equation*}
   \forall x\in[0,1), \quad v_s(x) = v\left(\frac{x}{1+s},\frac{1-x}{1+s},\frac{s}{1+s}\right) = 1 - \exp\left(-T\left(\frac{x}{1+s},\frac{1-x}{1+s},\frac{s}{1+s}\right)\right)
 \end{equation*}
 where $T\left(\frac{x}{1+s},\frac{1-x}{1+s},\frac{s}{1+s}\right)$ satisfies
 \begin{equation}
 \label{eq:rappelDefT}
 x \left(1 - \exp\left(-T\left(\frac{x}{1+s},\frac{1-x}{1+s},\frac{s}{1+s}\right)\right)\right) + (1 - x) T\left(\frac{x}{1+s},\frac{1-x}{1+s},\frac{s}{1+s}\right)= s
 \end{equation}
This function is extended by continuity at point $x=1$ by $v_{s}(1)=\min(s,1)$. This function behaves nicely, at least if $s<1$.
\begin{lmm}
  [Properties of $v_s$]
  \label{lem:propVs}\ 

  \begin{itemize}
  \item For all $s\in \xR_{+}$, for all $ x\in [0,1]$, $1-e^{-s} \leq v_{s}(x)\leq \min(s,1)$,
  \item For all $s\in \xR_{+}$, for all $ x\in [0,1)$, $v_{s}(x) < \min(s,1)$ and $v_{s}'(x) >0$,
  \item For all $s <1$, for all $ x\in [0,1]$, $(1-s)(e^{-s} + s-1)\leq v_{s}'(x) \leq\frac{e^{-s}(-s - \log(1-s))}{(1-s)}$, 
  \item For all $s <1$, for all $ x\in [0,1]$, $s(1-s)^2(e^{-s} + s-1)\leq v_{s}''(x) \leq\frac{2 s e^{-s}(-s - \log(1-s))}{(1-s)^3}$.
  \end{itemize}
\end{lmm}

\begin{proof}
From now and only in this proof, we write $v$ (resp. $T$) instead of $v_s(x)$ (resp. $T\left(x/1+s,1-x/1+s,s/1+s\right)$)  for convenience. Since $1-e^{-t} \leq t$ for all $t\in\xR_{+}$, 
\[
  x(1-e^{-T})+(1-x)T \leq T
\]
which implies $s\leq T$ by the definition~\eqref{eq:rappelDefT} of $T$. 
Then we have
\[
  xs+(1-x)T \geq xs+(1-x)s = s = xv+(1-x)T,
\]
which implies $v\leq s$ for $x>0$. Since $T(0,1/1+s,s/1+s)=s$, we have
$v_{s}(0) = 1-e^{-s} \leq s$. Moreover $s\leq T$ obviously implies
$1-e^{-s} \leq1-e^{-T}=v$, proving the first point. Finally if $x<1$ then $T>0$ thus
$s < T$ and $v_{s}(x)<s$ for all $x\in [0,1)$.

Rewriting \eqref{eq:rappelDefT} in terms of $v$ yields the 
relation
\begin{equation}
  \label{eq:bienPratique}
  xv-(1-x)\log(1-v) = s.
\end{equation}
Differentiating this formula and isolating $v'$ one gets
\( v' = \frac{ (- v - \log(1 - v))(1-v)}{1-xv}\); using \eqref{eq:bienPratique}
to get rid of the logarithm yields
\begin{equation}
  \label{eq:vprime}
  v'   = \frac{(1-v)}{1-xv}\cdot \frac{s-v}{1-x},
\end{equation}
proving $v_{s}'(x)>0$ for all $x\in[0,1)$.


Since $t\mapsto -t-\log(1-t)$ is nondecreasing on $[0,1]$, we obtain for $0<s<1$
\[
  \frac{(1-s)(e^{-s} + s-1)}{(1-x+xe^{-s})}\leq v_{s}'(x) 
  \leq \frac{e^{-s}(-s - \log(1-s))}{(1-xs)}
\]
which gives the following bounds (uniformly with respect to $x$):
\[
  (1-s)(e^{-s} + s-1) \leq v_{s}'(x) \leq \frac{e^{-s}(-s - \log(1-s))}{(1-s)}
\]
Let us now turn to   the second derivative. 
Take the logarithm of~\eqref{eq:vprime} and differentiate:
\[
  [\log(v')]' = \frac{v''}{v'} 
  = \frac{-v'}{1-v}+\frac{-v'}{s-v}+\frac{v}{1-xv} + \frac{xv'}{1-xv}+\frac{1}{1-x}.
\]
Using the expression~\eqref{eq:vprime} for $v'$, it is easy to see that the sum
of the first and fourth terms is $(-(s-v))/(1-xv)^2$ and the sum of the second
and fifth terms is $v/(1-xv)$. Therefore the whole sum is simply 
\[
  \frac{v''}{v'} 
  = - \frac{s-v}{(1-xv)^2} + \frac{v}{1 - xv} + \frac{v}{1-xv}
  = \frac{-2xv^2 + 3v -s}{(1-xv)^2}.
\]
From this expression and the bounds on $v'$,  the upper bound
on $v''$ is easily obtained. Moreover, since
\[
  3v-2xv^2-s -s(1-s) \geq -2(v-3/4)^2+(s-1)^2+1/8
\]
the lower bound $s(1-s)$  on $v''/v'$ will follow if we show that
$-2(v-3/4)^2+(s-1)^2+1/8$ is positive. This quantity is minimal if
 $|v-3/4|$ is maximal. Since $v$ is nondecreasing, the maximal value
of $|v-3/4|$ is attained at $x=0$ or $x=1$, for which $v=1-e^{-s}$ or
$v=s$. When $v=1-e^{-s}$, we have $3v-2v^2-s -s(1-s) =
3-3e^{-s}-2(1-e^{-s})^2-s-s(1-s)$ which is positive for all $s<1$.
When $v=s$, we have $3v-2v^2-s-s(1-s) = s(1-s)\geq0 $. This concludes the proof 
of the lower bound for $v''$. 
\end{proof}

Now we are able to prove that  the martingale problem is well posed by proving
pathwise uniqueness thanks to the following theorem of Yamada and Watanabe, 
as stated in~\cite[Theorem 5.3.3]{Dur96}. 
\begin{thrm}[Yamada-Watanabe]
Let \( dX_t = \sqrt{a(x)} dB_t + b(X_t)dt\) be a SDE such that
\begin{enumerate}[label={(\roman*)}]
\item there exists a positive increasing function $\rho$ on $(0,+\infty)$ such that
\[
\left|\sqrt{a(x)} -\sqrt{a(y)}\right| \leq \rho(|x-y|), \qquad \text{ for all } x,y \in \xR
\]
and
\(
\int_{]0,1[} \rho ^{-2}(u) du = +\infty.
\)
\item there exists a positive increasing concave function $\kappa$ on $(0,+\infty)$ such that
\[
|b(x)-b(y) | \leq \kappa(|x-y|), \qquad \text{ for all } x,y \in \xR
\]
and
\(
\int_{]0,1[}\kappa^{-1}(u) du = +\infty.
\)
\end{enumerate}
Then pathwise uniqueness holds for the SDE.
\end{thrm}

For $\sqrt{a}$, thanks to the previous bounds in Lemma~\ref{lem:propVs} and the 
elementary inequality $\abs{\sqrt{c} - \sqrt{d}} \leq \sqrt{\abs{c-d}}$, valid for all
$(c,d)\in\xR_+^2$, 
there exists a constant $C$ depending only on $s$ such that 
\begin{align*}
\left|\sqrt{a(x)}-\sqrt{a(y)} \right| &\leq \left|\sqrt{\frac{x(1-x)}{v_{s}(x)}} - \sqrt{\frac{y(1-y)}{v_{s}(x)}} \right| + \left|\sqrt{\frac{y(1-y)}{v_{s}(x)}} - \sqrt{\frac{y(1-y)}{v_{s}(y)}} \right|\\
&\leq \frac{\left|\sqrt{x(1-x)} - \sqrt{y(1-y)} \right|}{\sqrt{v_{s}(x)}} + \sqrt{\frac{y(1-y)}{v_{s}(x)v_{s}(y)}}\left|\sqrt{v_{s}(y)} -\sqrt{v_{s}(x)}\right|\\
&\leq \frac{\left|\sqrt{x(1-x)} - \sqrt{y(1-y)} \right|}{\sqrt{v_{s}(x)}} + \sqrt{\frac{y(1-y)}{v_{s}(x)v_{s}(y)}}\sup_{z\in[0,1]}\frac{|v_{s}'(z)|}{2\sqrt{v_{s}(z)}}\left|x-y\right|\\
&\leq C \sqrt{|x-y|}\sqrt{|1-x-y|} + C \left|x-y\right|\leq 2C \sqrt{|x-y|}.
\end{align*}
Thus the first item holds with $\rho(u) = 2C\sqrt{u}$.
 For the drift $b$, we have
\begin{align*}
|b(x)-b(y)| 
&\leq \beta \left|x(1-x)-y(1-y)\right| 
  + \left| x(1-x)\frac{v_{s}'(x)}{v_{s}^{2}(x)} 
                      - y(1-y)\frac{v_{s}'(y)}{v_{s}^{2}(y)} \right| 
\\
&\leq\beta |x-y||1-x-y| + \frac{v_{s}'(x)}{v_{s}^{2}(x)}\left|x(1-x) - y(1-y) \right|
 + y(1-y)\left|
             \frac{v_{s}'(x)}{v_{s}^{2}(x)} - \frac{v_{s}'(y)}{v_{s}^{2}(y)}
	   \right| 
\\
&\leq (\beta+C) |x-y| 
  + y(1-y) \sup_{z\in[0,1]}\frac{
      |v''_{s}(z)v_{s}^{2}(z)-2v_{s}'(z)v_{s}(z)v_{s}'(z)|
    }{
      v_{s}^{4}(z)}
  |x-y| \\
&\leq(\beta+2C) |x-y|,
\end{align*}
which proves the second item by setting $\kappa(u) = (\beta+2C)u$.

\subsection{The number of reproductions}  
\label{sec:aboutXTilde}
In this section we use the results from the previous section to study the
``first step''  of each generation, 
getting information on the asymptotics of the random number of
reproductions. 
\paragraph{Notation.} From now on, we only need to study what happens
in one step of the Markov chain. 
We let  $x=X^n_0 = w/n \in \stateSpace$ be the initial proportion of white balls. 
There are $ b$ black balls, where $b+w = n$ and we draw $f = s n$ times. 
Note that $N = w+b+f = (1+s)n$, and $s$ is fixed, so $n$ and $N$ are of the 
same order. 
We omit the ``size'' index $n$ and the time index $k=1$,  denoting by $(\tX,\tY) =
(\tX^n_1, \tY^n_1)$ the number of white/black ``reproductions'' and by $X = X^n_1$
the proportion of white balls after the first step. 
Moreover we let
\[ \tx = \espX{\tX}, \qquad \ty = \espX{\tY}.\]
The goal of this section is to prove the following estimates. 
 \begin{proposition}[Moments of $(\tX,\tY)$]
   \label{prop:momentsTilde}
   The moments of $(\tX,\tY)$ have the following asymptotic behaviour:
   \begin{align*}
     \tx = \espX{\tX} &=  nxv_s(x) + \bigO(1),    &
     \ty = \espX{\tY} &=  n(1-x)v_s(x) + \bigO(1),\\
     \varX{\tX}       &=  \bigO(n), &
     \varX{\tY}       &= \bigO(n),\\
   \covX{\tX,\tY}     &= \bigO(n), & 
   \end{align*}
   Moreover, 
   \begin{align*}
   -b\varX{\tX} + (w-b)\covX{\tX,\tY} + w\varX{\tY} &= 
   wb\left(v'_s(x) (v_s(x)-1)  + \bigO(1/n)\right) \\
   b^2 \varX{\tX}  - 2wb\covX{\tX,\tY} + w^2 \varX{\tY} &= 
   wbn\left( v_s(x)(1 - v_s(x)) + \bigO(1/n)\right).
 \end{align*}
 In all these results the ``$\bigO$'' are uniform  on 
 the starting point $x\in[0,1 - y_0]$ (if $s\geq 1$), 
 and uniform on $x\in[0,1]$ (if $s<1$). 
 Finally, 
 \begin{align*}
   \espX{\abs{\tX - \esp{\tX}}^3} & = \bigO(n^{3/2}) \\
   \espX{\abs{\tY - \esp{\tY}}^3} & = \bigO(n^{3/2}), 
 \end{align*}
 where the $\bigO$ are uniform on $x\in[0,1]$. 
 \end{proposition}
 \begin{remark}
   Getting the exact value of the leading term for the second moments
   does not seem easy; 
   we will only need a control on the  particular linear combinations that
   appear in the second block of equations.
 \end{remark}

We begin with a lemma.  Define $p_{ww}$ to be the probability that two
given different white balls are drawn, and define $p_{wb}$, $p_{bb}$ similarly.
 \begin{lmm}
   \label{lmm:asymptotique_pww}
   For all $y_0$, there exists a $C(y_0)$ such that, if $(w,b,f) \in \Omega_N(y_0)$,
   \begin{align*}
     \abs{(p_{ww} - v^2)(w,b,f)} \leq \frac{C(y_0)}{N}, 
     \abs{(p_{wb} - v^2)(w,b,f)} \leq \frac{C(y_0)}{N}, 
     \abs{(p_{bb} - v^2)(w,b,f)} \leq \frac{C(y_0)}{N}. 
   \end{align*}
   Moreover, the differences are of order $1/N$ and are given by:
   \begin{align*}
     \abs{
     \left(p_{bb} - p_{wb} - \frac{1}{N}\left( (2 v - 1) (\partial_x-\partial_y)v\right)^N\right)(w,b,f)
   } &\leq \frac{C(y_0)}{N^2},  \\
   \abs{
     \left(p_{wb} - p_{ww} - \frac{1}{N}\left( (2 v - 1) (\partial_x-\partial_y)v\right)^N\right)(w,b,f)
   } &\leq \frac{C(y_0)}{N^2}.
   \end{align*}
   The same bounds bold uniformly on $\Omega_N(s)$ if $s<1$. 
 \end{lmm}

 \begin{proof}

To compute these quantities, let $q_{ww} = \prb{\text{neither }B_1\text{ nor
}B_2\text{ are drawn}}$, and define $q_{wb}$, $q_{bb}$ similarly. 
As before, the probability $q_{ww}$ does not depend on the color of the two balls, 
but only on the composition of the remainder of the urn. In terms of the
quantity $\tq$ introduced in Definition~\ref{dfn:qTilde}, we have:
\begin{align*}
  q_{ww}(w,b,f) &= \tq(w-2,b,f)   &
  q_{wb}(w,b,f) &= \tq(w-1,b-1,f), \\
  q_{bb}(w,b,f) &= \tq(w,b-2,f). &
\end{align*}

 Going back to the probabilities of reproduction is easy. Since for any events $A$ and $B$, 
 \( \prb{A^c \cap B^c} + \prb{A} + \prb{B} = 1 + \prb{A\cap B}, \)
 we get
 \begin{align*}
   p_{ww} &= q_{ww} - 1 + 2p_w \\
   p_{wb} &= q_{wb} - 1 + p_w + p_b \\
   p_{bb} &= q_{bb} - 1 + 2p_b. 
 \end{align*}
 The result follows using the previous approximations on $p_w$ and $p_b$ from Theorem~\ref{thm=approximation}
 and the results on $\tq$ (Proposition~\ref{prop:qTilde}).
 \end{proof}

 \begin{proof}[Proof of Proposition~\ref{prop:momentsTilde}]
   
The variance and covariance of $\tX$ and $\tY$ are easily computed in terms of these quantities. 
\begin{equation}
  \label{eq=varcovXtilde}
\begin{aligned}
  \varX{\tX} &= \var{ \sum_{i=1}^w \ind{B_i}} 
 = w \varX{\ind{B_1}} + w(w-1) \covX{\ind{B_1},\ind{B_2}} \\
          &= wp_w(1 - p_w) + w(w-1)(p_{ww} - p_w^2) \\
	  &= w(p_w - p_{ww}) + w^2(p_{ww} - p_w^2)  \\
 \covX{\tX,\tY}&= wb(p_{wb} - p_wp_b) \\
 \varX{\tY} &= b(p_b - p_{bb}) + b^2(p_{bb} - p_b^2) 
\end{aligned}
\end{equation}
   
   Since all expectations are taken starting from $x$, we drop the
   subscript $x$ in the proofs.  By \eqref{eq=varcovXtilde}, all the quantities
   considered may be expressed in terms of $p_w$, $p_b$, $p_{ww}$, $p_{bb}$ and
   $p_{wb}$.  Using the asymptotic results from Theorem~\ref{thm=approximation}
   and Lemma~\ref{lmm:asymptotique_pww}, we get the results after a
   short computation.  For example, the last result follows from:
   \begin{align*}
   &b^2 \var{\tX}  - 2wb\cov{\tX,\tY} + w^2 \var{\tY} \\
   &= b^2w(p_w - p_{ww}) + b^2w^2(p_{ww} - p_w^2) -2w^2b^2(p_{wb} - p_w p_b)              \\
   &\quad + w^2b (p_b - p_{bb}) + w^2b^2 (p_{bb} - p_b^2)       \\
   &= wb\left[ b(p_w - p_{ww}) + w (p_b - p_{bb}) +  wb\left[ p_{ww} - p_w^2 - 2p_{wb} + 2p_wp_b + p_{bb} - p_b^2
   \right] \right] \\
   &= wb\left[ (b+w)v_s(x)(1-v_s(x)) + \bigO(1) +  wb\left[ p_{ww} + p_{bb} - 2p_{wb} -  (p_w - p_b)^2
   \right] \right] \\
   &= wb\left[ (b+w)v_s(x)(1-v_s(x)) + \bigO(1)\right]
   \end{align*}
   where in the last line, we have used the fact that $p_{bb} - p_{wb}$ and
   $p_{ww} - p_{wb}$ have the same leading term at order $1/n$, so that they
   cancel out. 

   The bounds on the higher order moments follow from the fact that 
   $\tX$ and $\tY$ are sums of negatively related indicator variables. 
   For example, $\tX$ is the sum of $w$ indicators, so that
   by the bound~\eqref{eq:momentOrdreTrois} of Theorem~\ref{thm:concentration}, 
   \[ 
     \esp{ \abs{\tX - \esp{\tX}}^3} \leq 12e w^{3/2} \leq C n^{3/2}. \qedhere
   \]
 \end{proof}

\subsection{Infinitesimal mean}
\label{sec:infMean}
From this moment on, until Section~\ref{sec:betaNonZero}, the parameter $\beta$
is fixed to $0$. We set
\[ \tZ = \frac{\tX}{\tX + \tY} = \phi(\tX,\tY) \]
where $\phi:(x,y)\mapsto x/(x+y)$. Let $\mathcal{F}$ be the sigma-field generated by 
$(\tX, \tY)$. Recall that (since $\beta = 0$), $X$ is the proportion of white balls
after a binomial sampling with probability $\tZ$. 
By conditioning, 
\begin{align*}
  \espX{X} 
  &= \espX{\espX{X \middle| \mathcal{F}}}  \\
  &= \espX{\tZ},
\end{align*}
so it makes sense to study the first moment of $\tZ$. 
\begin{lmm}
  [Expectation of $\tZ$]
  \label{lmm:expZTilde}
  The first moment of $\tZ$ is given by
\begin{equation}
  \label{eq:asymptoticsZTilde}
  \esp{\tZ} - x = - \frac{1}{n}\cdot \frac{x(1-x)v'_s(x)}{v_s^2(x)} + \bigO(1/n^{3/2}).
\end{equation}

\end{lmm}
\begin{crllr}
  The formula \eqref{eq:infMean} holds when $\beta = 0$. 
\end{crllr}
\begin{proof}
  The proportion $\tZ$ is a function of $\tX$ and $\tY$. 
  We wish to apply Taylor's formula to $\phi$ to compare $\espX{\tZ}$ to $\phi\left(\espX{\tX}, \espX{\tY}\right) = 
\phi(\tx,\ty)$. The derivatives of $\phi$ are given by:
\begin{equation}
  \label{eq=deriveesPhi}
\begin{aligned}
  \partial_1\phi(x,y) &= \frac{y}{(x+y)^2}, &
  \partial_2\phi(x,y) &= \frac{-x}{(x+y)^2} \\
  \partial_{11}\phi (x,y) &= \frac{-2y}{(x+y)^3} &
  \partial_{22}\phi (x,y) &= \frac{2x}{(x+y)^3} \\
  \partial_{12}\phi (x,y) &= \frac{x-y}{(x+y)^3}. &
\end{aligned}
\end{equation}

Let us apply Taylor's formula to $\phi$ on the segment $S = [(\tx,\ty), (\tX,\tY)]$.  
\begin{equation}
  \label{eq:TaylorForPhi}
  \phi(\tX,\tY) - \phi(\tx,\ty) = T_1 + T_2 + T_3
\end{equation}
where
\begin{align*}
  T_1 &= \partial_1\phi(\tx,\ty)( \tX - \tx) + \partial_2\phi(\tx,\ty)(\tY - \ty), \\
  T_2 &= \frac{1}{2}\partial_{11}\phi(\tx,\ty)(\tX - \tx)^2
         + \partial_{12}\phi(\tx,\ty)(\tX - \tx)(\tY - \ty)
         + \frac{1}{2}\partial_{22}\phi(\tx,\ty)(\tY - \ty)^2, 
\end{align*}
and $T_3$ is a remainder term which will be considered later. 

We take expectations on both sides, once more dropping the subscript $x$
from the notation.  The first-order term $T_1$ disappears, so
\begin{equation}
  \label{eq:TMinusOne}
  \espX{\tZ} - x = \esp{\phi(\tX,\tY)} - x = (\phi(\tx,\ty) - x) + \esp{T_2} + \esp{T_3}.
\end{equation}
Let us look at these three terms in turn. 
For the first one:
\begin{align}
  \phi(\tx,\ty) - x
  &= \frac{wp_w}{w p_w + b p_b} - \frac{w}{w+b} 
  \notag\\
  &= x\left( \frac{ p_w - xp_w - (1-x)p_b}{xp_w + (1-x)p_b} \right)
  \notag\\
  &= x(1-x)\frac{p_w - p_b}{xp_w + (1-x)p_b} 
  \notag\\
  \label{eq:TZero}
  &= -\frac{1}{n}x(1-x)\frac{v_s'(x)}{v_{s}(x)} + \bigO(1/n^2),
\end{align}
where we used Theorem~\ref{thm=approximation} and the  fact that $N = (1+s)n$ in the last line.

The second term $T_2$ is a bit trickier. 
\[
  \esp{T_2} =  \frac{1}{2}\partial_{11}\phi(\tx,\ty)\var{\tX} + \partial_{12}\phi(\tx,\ty)\cov{\tX,\tY}
  +\frac{1}{2}\partial_{22}\phi(\tx,\ty) \var{\tY}.
\]
First, remark that $\partial_{11}\phi(\tx,\ty) = \frac{1}{n^2v_s(x)^2}\partial_{11}\phi(x,1-x) + \bigO(1/n^3)$, 
and that similar results hold for the other derivatives, so that, using the rough bounds
on the variances from Proposition~\ref{prop:momentsTilde}, we get 
\begin{align*}
  \esp{T_2}
  &=  \frac{1}{n^2v_s(x)^2} \left(\frac{1}{2}\partial_{11}\phi(x,1-x)\var{\tX} + \partial_{12}\phi(x,1-x)\cov{\tX,\tY} \right. \\
  &\qquad \left. +\frac{1}{2}\partial_{22}\phi(x,1-x) \var{\tY}\right) + \bigO(1/n^2). 
\end{align*}
Due to the explicit expression of the derivatives (equation~\eqref{eq=deriveesPhi}), the term between brackets is, 
up to a factor $n$, the one 
that appears in Proposition~\ref{prop:momentsTilde}, so
\begin{equation}
  \label{eq:TTwo}
  \esp{T_2} = 
  \frac{1}{n} x(1-x) \frac{ v'_s(x)(v_s(x)-1)}{v_s(x)^2} + \bigO(1/n^2).
\end{equation}
We will prove in the next section that  $T_3 = \bigO(1/n^{3/2})$. 
Inserting \eqref{eq:TZero} and \eqref{eq:TTwo}
in \eqref{eq:TMinusOne}, we obtain~\eqref{eq:asymptoticsZTilde}:
\begin{equation*}
  \esp{\tZ} - x = - \frac{1}{n}\cdot \frac{x(1-x)v'_s(x)}{v_s^2(x)} + \bigO(1/n^{3/2}).
\end{equation*}
Multiplying  by $n$, we get~\eqref{eq:infMean}, which proves the corollary.
\end{proof}

\subsection{The remainder}
\label{sec:remainder}
Let us bound the remainder term $T_3$ in Taylor's formula~\eqref{eq:TaylorForPhi}.
If we let $z_0 = (x_0,y_0) = (\tx,\ty)$ and $z_1 =
(x_1,y_1) = (\tX,\tY)$ ($z_0$ is fixed and $z_1$ is random), then
$T_3$ can be written as:

\[
  T_3 = \sum_{\alpha, \abs{\alpha} = 3} \frac{3}{\alpha !}  \int_0^1(1-t)^2 D_\alpha \phi(z_t) dt \cdot  (z_1-z_0)^\alpha,
\]
where $z_t = (1-t)z_0 + tz_1$. 
To get the claimed bound on $\esp{T_3}$, it suffices to show that, for any multi-index $\alpha$ of length $3$,
\[
  R_\alpha = \esp{ \int_0^1 \abs{\partial_\alpha \phi (z_t)} dt \abs{ (z_1 - z_0)^\alpha}}  = \bigO(1/n^{3/2}).
\]
Therefore we have to bound the third derivatives of $\phi$
on the segment $[z_0,z_1]$. 
The difficulty here is that $(\tX,\tY)$ may be very close to $(0,0)$, where the 
derivatives of $\phi$ blow up. This problem only occurs if both $\tX$ and $\tY$ 
are small. However, if $x \geq 1/2$, $\tX$ is unlikely to be small, 
and if $x < 1/2$ then $\tY$ should not be too small. This prompts us to introduce the following good event:
\begin{equation}
  A = \begin{cases} 
    \{ \tX \geq \tx/2\} & \text{ if } x\geq 1/2, \\
    \{ \tY \geq \ty/2\} & \text{ if } x <   1/2.
  \end{cases}
  \label{eq:defGoodEvent}
\end{equation}
\paragraph{Step 1: a bad bound on the bad event}
All third derivatives of $\phi$ satisfy: 
\[
  \abs{\partial_\alpha  \phi(x,y) } \leq \frac{C}{(x+y)^4} (\abs{x} + \abs{y}).
\]
Since $1 \leq \tX + \tY\leq n$ (at least one ball is chosen),
$1\leq \tx + \ty = \esp{\tX + \tY} \leq n$, so that for all $t\in[0,1]$,
\begin{equation}
  \label{eq:crudeBound}
  \abs{\partial_\alpha  \phi(z_t) } \leq Cn.
\end{equation}
This bound is not strong but it holds even on the ``bad event'' $A^c$. 

\paragraph{Step 2: a good bound on the good event}
Suppose $x\geq 1/2$, so that on $A$, $\tX \geq \tx/2$. 
By the asymptotic result of Proposition~\ref{prop:momentsTilde} on $\tx$, 
and the fact that $v_s$ is bounded below by $s$, 
\begin{align*}
  \forall t, \quad \abs{\partial_\alpha \phi(z_t) }
  \leq \frac{C n}{(1/2)^4 (\tx)^4} 
  \leq \frac{C'}{n^3}.
\end{align*}
If $x<1/2$, $\tY \geq \ty/2$ on $A$, so 
\begin{align*}
  \forall t, \quad \abs{\partial_\alpha \phi(z_t) }
  \leq \frac{C n}{(1/2)^4 (\ty)^4} 
  \leq \frac{C'}{n^3}.
\end{align*}

\paragraph{Step 3: the good event has very high probability}
Suppose first that $x\geq 1/2$, so $A = \{\tX  \geq \esp{\tX}/2\}$. Intuitively, since $\tx = \esp{\tX}$ is of order $n$ and its 
standard deviation is of order $\sqrt{n}$, the event $\tX \leq \tx/2$ should have very small
probability. This  rigorous proof follows from the deviation bounds established above. Indeed
\begin{align*}
  \prb{A^c} &\leq \prb{ \abs{\tX  - \esp{\tX}} > \esp{\tX}/2} \\
  &\leq \exp\PAR{ - \frac{\esp{\tX}^2}{16 xn} }
\end{align*}
where we used \eqref{eq:concentration} applied to $\tX$, a sum of $xn$ negatively correlated indicator. 
Since $x\geq 1/2$, for some absolute constant $C$ we find that
\begin{equation}
  \label{eq:AtresImprobable}
  \prb{A^c}  \leq \exp\PAR{ - Cn}. 
\end{equation}
If $x<1/2$, $A = \{ \tY \geq \esp{\tY}/2\}$, and we can apply \eqref{eq:concentration} to $\tY$
to see that \eqref{eq:AtresImprobable} still holds. 

\paragraph{Step 4: conclusion}
For any multiindex $\alpha$, using the good bounds on the event $A$, 
and the crude bounds \eqref{eq:crudeBound} and $\abs{\tX - \tx} \leq 2n$, we get:
\begin{align*}
R_\alpha
&= \esp{ \int_0^1 \abs{\partial_\alpha \phi (z_t)} dt \abs{ (z_1 - z_0)^\alpha}} \\
&= \frac{C}{n^3} \esp{\ind{A} \abs{(\tX - \tx)^\alpha}}
   + Cn^4 \exp( - Cn). 
\end{align*}
The third moment bounds from Proposition~\ref{prop:momentsTilde} yield:
\begin{align*}
R_\alpha
&= \frac{C}{n^{3/2}} + Cn^4 \exp( - Cn)   = \bigO(n^{-3/2}). 
\end{align*}
This proves that $\esp{T_3} = \bigO(n^{-3/2})$, and concludes the proof of
the formula \eqref{eq:infMean} for the infinitesimal mean in the case $\beta =0$.

\subsection{Infinitesimal variance} 
\label{sec:infVariance}
Let us first study the second moment of $\tZ$. 
\begin{lmm}[Variance of $\tZ$]
  \label{lmm:varianceZTilde}
  The variance of $\tZ$ is given by
  \begin{equation}
    \label{eq:varianceZTilde}
   \var{\tZ}
   = 
   \frac{1}{n} x(1-x)\frac{v_s(x)(1-v_s(x))}{v_s(x)^2} + \bigO(1/n^{3/2}).
 \end{equation}
\end{lmm}
\begin{proof}
  Since we already know the behaviour of $\esp{\tZ}$, we only
  need  to compute $\espX{\tZ^2}$. To this end let $\psi(x,y) = \phi(x,y)^2$, so that
$\espX{\tZ^2} = \espX{\psi(\tX,\tY)}$ and we can use  Taylor's formula once
again: $\psi(\tX,\tY) - \psi(\tx,\ty) = T'_1 + T'_2 + T'_3$,
where $T'_i$ is the $i$\textsuperscript{th} order term. 
Taking expectations (dropping once more the subscript $x$) yields:
\[ \esp{\tZ^2} = \psi(\tx,\ty) + \esp{T'_2} + \esp{T'_3}. \]
The term $T'_3$ is treated as before to get:
\[ \esp{T'_3} = \bigO(1/n^{3/2}). \]

To treat $T'_2$ we compute the derivatives of $\psi$:
\begin{equation}
  \label{eq=deriveesPsi}
\begin{aligned}
  \partial_1\psi(x,y) &= \frac{2xy}{(x+y)^3}, &
  \partial_2\psi(x,y) &= \frac{-2x^2}{(x+y)^3} \\
  \partial_{11}\psi (x,y) &= \frac{2y(y-2x)}{(x+y)^4}, &
  \partial_{22}\psi (x,y) &= \frac{6x^2}{(x+y)^4}, \\
  \partial_{12}\psi (x,y) &= \frac{2x(x-2y)}{(x+y)^3}. &
\end{aligned}
\end{equation}
Therefore $\esp{T'_2}$  is given by:
\[
  \esp{T'_2} = 
       \frac{1}{2} \partial_{11}\psi(\tx,\ty) \var{\tX} 
       +           \partial_{12}\psi(\tx,\ty) \cov{\tX,\tY}
       + \frac{1}{2} \partial_{11}\psi(\tx,\ty) \var{\tY}.
\]
As before, we can approximate the derivatives at $(\tx,\ty)$ by the ones at $(x,y)$ since
\[ 
  \abs{\partial_{11}\psi (\tx,\ty)
  - \frac{1}{n^2v_s(x)^2}\partial_{11}\psi (x,1-x)%
  } \leq \bigO(1/n^3).
\]
Using the expressions of the partial derivatives of $\psi$, and rearranging
terms to use the results of Proposition~\ref{prop:momentsTilde}, we get:
\begin{align*}
  &\esp{T'_2}\\
  &= 
       \frac{1}{n^2v_s(x)^2} \left(
       \frac{1}{2}(2y(y-2x)) \var{\tX} 
       +           (2x(x-2y)) \cov{\tX,\tY}
       + \frac{1}{2} 6x^2 \var{\tY}\right)         \\
  &= 
       \frac{2x}{n^2v_s(x)^2} \left(
       -y\var{\tX} + (x-y)\cov{\tX,\tY} + x \var{\tY} \right) \\
  &\qquad +
       \frac{1}{n^2v_s(x)^2} \left(
       y^2\var{\tX} -2xy\cov{\tX,\tY} + x^2 \var{\tY} \right) \\
       &= \frac{x(1-x)}{nv_s(x)^2} \left( 2xv'_s(x)(v_s(x) - 1)
             + v_s(x)(1 - v_s(x))  + \bigO(1/n) \right).
\end{align*}

The last step is to find the behaviour of $\psi(\tx,\ty)$. Once more, by Taylor's formula:
\begin{align*}
  \psi(\tx,\ty) 
  &= \psi(x,(1-x)\frac{p_b}{p_w}) \\
  &= \psi(x,(1-x)(1+ \frac{p_b-p_w}{p_w}))  \\
  &= \psi(x,(1-x)) + \partial_2\psi(x,1-x) \cdot (1-x)\frac{v'_s(x)}{nv_s(x)}  + \bigO(1/n^2) \\
  &= x^2  - \frac{1}{n} 2x^2(1-x)\frac{v'_s(x)}{v_s(x)} + \bigO(1/n^2). 
\end{align*}

Therefore:
\[ \esp{\tZ^2} = x^2 + \frac{x(1-x)}{nv_s(x)^2}\left(
   -2xv'_s(x) + v_s(x)(1-v_s(x)) + \bigO(1/n)\right) .
\]
Since, by \eqref{eq:asymptoticsZTilde}, 
\[
   \esp{\tZ} = x -  \frac{1}{n} x(1-x) \frac{v'_s(x)}{v_s(x)^2} + \bigO(1/n^{3/2}),
\]
we finally obtain the expression \eqref{eq:varianceZTilde} for the variance of $\tZ$.
\end{proof}
\begin{proof}
  [Proof of the formula \eqref{eq:infVar} for the infinitesimal variance]
Since $\beta = 0$, the conditional law of $nX$ given $\mathcal{F}$ is the binomial law $\mathcal{B}(n,\tZ)$. 
The variance of $X$ is given by conditioning on $\mathcal{F}$:
\begin{align*}
  \varX{X}
  &= \espX{\varX{X \middle| \mathcal{F}}} + \varX{\espX{X \middle| \mathcal{F}}} \\
  &= \espX{ \frac{1}{n}\tZ(1-\tZ)} + \varX{\tZ}.
\end{align*}
We rearrange terms on the right hand side to get:
\begin{align*}
  \varX{X}
  &= \frac{1}{n} \espX{\tZ} - \frac{1}{n} \espX{\tZ^2} + \espX{\tZ^2} - \espX{\tZ}^2 \\
  &= (1-1/n) \varX{\tZ}  + \frac{1}{n} \espX{\tZ}\left(1 - \espX{\tZ}\right).
\end{align*}
The asymptotics of $\esp{\tZ}$ and $\varX{\tZ}$ are known from lemmas~\ref{lmm:expZTilde}
and~\ref{lmm:varianceZTilde}. Injecting them in the last equation we get
\begin{align*}
  \varX{X}
  &=\frac{1}{n}x(1-x) \frac{v_s(x)(1-v_s(x))}{v_s^2(x)}   + \frac{1}{n}x(1-x) + \bigO(1/n^{3/2}).
\end{align*}
 The infinitesimal variance is obtained by multiplying by $n$:
 \begin{align*}
   a_n(x) &= x(1-x)\frac{v_s(x)(1-v_s(x))}{v_s^2(x)} + x(1-x) + \bigO(1/n^{1/2}) \\
   &= \frac{1}{v_s(x)} x(1-x) + \bigO(1/n^{1/2}).
   \qedhere
 \end{align*}
\end{proof}

\subsection{No jumps at the limit}
\label{sec:tightness}
We have to show that
\(
  n\espX{ \abs{X - x}^3} \xrightarrow{} 0. 
  \)
Recalling that 
$\tZ = \phi(\tX,\tY) = \tX/(\tX + \tY)$, and using
the trivial bound $(a+b+c)^3 \leq 9(a^3 + b^3 + c^3)$, 
\begin{equation}
  \label{eq:decompositionTightness}
  n \espX{ \abs{X - x}^3}
  \leq 
  9n  \espX{ \abs{X - \tZ}^3}  + 9 n\espX{\abs{\tZ - \frac{\tx}{\tx + \ty}}^3} + 9n\PAR{\frac{\tx}{\tx+\ty} - x}^3.
\end{equation}
For the first term, we condition by the first step:
 \[
   \espX{ \abs{X - \tZ}^3} = \espX{ \esp{\abs{X - \tZ}^3 \middle| \mathcal{F}}}.
 \]
Given $\mathcal{F}$, $n X$ follows a binomial law of parameters $xn$ and 
$\tZ$. Using (for example) the bound \eqref{eq:momentOrdreTrois} (which holds in the more general
negatively dependent case), the whole term is~$\bigO(n^{-1/2})$.

The third term of \eqref{eq:decompositionTightness} is $\bigO(1/n^2)$, thanks to the 
controls on $\tx$ and $\ty$ from Proposition~\ref{prop:momentsTilde}.

Therefore we only have to bound the second term $n\espX{\abs{\tZ - \frac{\tx}{\tx + \ty}}}$. 
Let us reuse the notation $z_t = (1-t) (\tx,\ty) + t (\tX,\tY)$: 
\begin{align*}
  n\espX{\abs{\tZ - \frac{\tx}{\tx + \ty}}^3}
  &= n\espX{\abs{\phi(z_1) - \phi(z_0)}^3}.
\end{align*}
As in Section~\ref{sec:remainder}, we want to bound the derivatives of $\phi$ on the segment $[z_0,z_1]$, 
which is only possible if $\tX + \tY$ is large enough. Recall the good event $A$ from
Equation~\eqref{eq:defGoodEvent}.  On~$A$, it is easy to see that all first derivatives of $\phi$ are bounded by $C/n$ for some absolute
constant $C$, therefore
\begin{align*}
  n\espX{\abs{\tZ - \frac{\tx}{\tx + \ty}}^3}
  &\leq  \frac{C}{n^2} \PAR{\esp{ (\tX - \tx)^3} +  \esp{(\tY - \ty)^3}} + C\prb{A^c} n^4. 
\end{align*}
The third moments are controlled by Proposition~\ref{prop:momentsTilde}, and the probability
of the bad event is exponentially small (by \eqref{eq:AtresImprobable}). This shows that 
\eqref{eq:noJumps} holds, and concludes the proof of the main result when $\beta = 0$.

\subsection{Proofs for the full model}
\label{sec:betaNonZero}
In this final section we show how to compute the infinitesimal
mean and variance in the general case, that is, we prove Proposition~\ref{prop:infMeanInfVar}
when $\beta$ is arbitrary. 
We still denote by $\tX$, $\tY$ the number of white/black reproductions, 
by $\mathcal{F} = \sigma(\tX,\tY)$ the corresponding $\sigma$-field,
and by $\tZ$ the ``raw'' ratio $\tZ = \frac{\tX}{\tX + \tY}$. We define
\[ \tZ_\beta = \frac{ (1 + \beta/n) \tX}{(1+\beta/n)\tX + \tY}, \]
so that, conditionally on $\mathcal{F}$, $nX$ follows a binomial law
of parameters $n$ and $\tZ_\beta$. 
This modified ratio is not far from $\tZ$:
\begin{align}
  \tZ_\beta
  \notag
  &= \frac{ (1+\beta/n) \tX}{ (1+\beta/n) \tX + \tY} 
   = \tZ \left( \frac{ 1+ \beta/n}{(1+\beta/n) \tZ + 1 - \tZ} \right)\\
   \notag
  &= \tZ (1+\beta/n) (1 - (\beta/n)\tZ + \bigO(1/n^2)) \\
  \label{eq:ZTildeBetaAndZTilde}
  &= \tZ \left(1 + (\beta/n) (1 - \tZ) + \bigO(1/n^2)\right),
\end{align}
where the $\bigO$ is uniform on $x$ and $\omega$ since $\tZ$ is bounded. 
Taking the expectation gives
\begin{align*}
  \esp{\tZ_\beta} 
  &= \esp{\tZ} + (\beta/n) \esp{\tZ(1-\tZ)} + \bigO(1/n^2) \\
  &= \esp{\tZ} + (\beta/n) \esp{\tZ}\PAR{1 - \esp{\tZ}} - (\beta/n) \var{\tZ} + \bigO(1/n^2). 
\end{align*}
Now let us recall the results from lemmas~\ref{lmm:expZTilde} and~\ref{lmm:varianceZTilde}:
\begin{align*}
  \esp{\tZ} - x &= - \frac{1}{n}\cdot \frac{x(1-x)v'_s(x)}{v_s^2(x)} 
                   + \bigO(1/n^{3/2}), \\
   \var{\tZ} &= \frac{1}{n} x(1-x)\frac{v_s(x)(1-v_s(x))}{v_s(x)^2} + \bigO(1/n^{3/2}).
 \end{align*}
 This immediately entails 
\[
  \esp{\tZ_\beta} - x = - \frac{1}{n}\cdot \frac{x(1-x)v'_s(x)}{v_s^2(x)} 
                         + x(1-x) \frac{\beta}{n} + \bigO(1/n^{3/2}), 
\]
and proves the general form of the infinitesimal mean
announced in Proposition~\ref{prop:infMeanInfVar}. 

For the variance, squaring~\eqref{eq:ZTildeBetaAndZTilde} and taking expectations 
gives
\[
  \esp{\tZ_\beta^2} = \esp{\tZ^2}  + 2(\beta/n) \esp{\tZ^2(1-\tZ)} + \bigO(1/n^2),
\]
therefore
\begin{align*}
  \var{\tZ_\beta} &= \var{\tZ} + \frac{2\beta}{n} \PAR{
  \esp{\tZ^2(1-\tZ)} - \esp{\tZ}\esp{\tZ(1-\tZ)}%
  }
  + \bigO(1/n^2) \\
  &= \var{\tZ} + \frac{2\beta}{n} \PAR{
    \esp{\tZ(1-\tZ)\left(\tZ - \esp{\tZ}\right)}%
  }
  + \bigO(1/n^2).
\end{align*}
The absolute value of the second term is bounded above by
\( \frac{2\beta}{n} \var{\tZ}^{1/2} \); since $\var{\tZ}$ is
of order $1/n$ we get
\[
  \var{\tZ_\beta} = \var{\tZ} + \bigO(1/n^{3/2}),
\]
which proves that the infinitesimal variance does not depend on $\beta$. 

Finally we have to show that the control on the jump sizes still holds.
Adding one more intermediate term in \ref{eq:decompositionTightness} we get:
  \begin{align*}
  n \esp{ \abs{X - x}^3}
  &\leq 
  16n \bigg(   \esp{ \abs{X - \tZ_\beta}^3}  
            +\esp{ \abs{\tZ_\beta - \tZ}^3} \\
  &\qquad	    + \esp{\abs{\tZ - \frac{\tx}{\tx + \ty}}^3} 
	    + \PAR{\frac{\tx}{\tx+\ty} - x}^3 \bigg)
	  \end{align*}
The first, third and fourth terms are treated exactly as in Section~\ref{sec:tightness}, 
since $\tZ_\beta$ is the conditional expectation of $X$ knowing $\mathcal{F}$. 
For the second term, recalling \eqref{eq:ZTildeBetaAndZTilde}, we get
\[
  \abs{\tZ_\beta - \tZ}^3 = \frac{\beta^3}{n^3} \left(\tZ(1 - \tZ) + \bigO(1/n)\right)^3,
\]
which implies that
\( n \esp{ \abs{\tZ - \tZ_\beta}^3}\) converges to zero.
This concludes the proof of Proposition~\ref{prop:infMeanInfVar} in the general case.

\paragraph*{Acknowledgements.}
This paper stems from discussions with the biologists F.-X. Dechaume-Moncharmont
and M.~Galipaud, who came up with the single generation model for ``indirect'' fitness. 
We thank them for stimulating discussions on the biological aspects 
of the problem. On the mathematical side, our special thanks go
to R.~Eymard for his input on the inverse numerical analysis problem; 
we also thank V.-C. Tran and D.~Coupier for interesting discussions.
Finally we thank J.E.~Taylor and an anonymous reviewer for many valuable comments
on a preliminary version of this work.

  \bibliographystyle{amsalpha}
  \bibliography{spite}

  \bigskip

{\footnotesize %
  \noindent Ludovic \textsc{Gouden\`ege},
e-mail: \texttt{goudenege(AT)math.cnrs.fr}

 \medskip

   \noindent Pierre-Andr\'e~\textsc{Zitt},
e-mail: \texttt{pierre-andre.zitt(AT)u-pem.fr}

 \medskip

 \noindent\textsc{LAMA, 
 Universit\'e de Marne-la-Vall\'ee,
 5, boulevard Descartes,
 Cit\'e Descartes - Champs-sur-Marne,
 77454 Marne-la-Vall\'ee Cedex 2, France.}

}
\end{document}